\documentclass[10pt]{article}
\usepackage{amscd,amssymb,amsmath,verbatim,amsthm}
\usepackage{hyperref}
\usepackage{times}
\usepackage{graphicx}
\usepackage{enumerate}
\usepackage[usenames]{color}

\definecolor{r}{rgb}{.6,0,.3}
\definecolor{gr}{rgb}{0,0.6,.3}

\newcounter{intro}

\newtheorem{thm}{Theorem}[section]
\newtheorem{lem}[thm]{Lemma}
\newtheorem{prop}[thm]{Proposition}
\newtheorem{cor}[thm]{Corollary}

\newtheorem{remark}[thm]{Remark}

\newtheorem{rem}[thm]{Remark}

\numberwithin{equation}{section}   

\newcounter{counteroman}

\newcommand{\cref}[1]{Corollary~\ref{#1}}

\newcommand{\R}{\mathbb{R}}

\newcommand{\bpm}{\begin{pmatrix}}
\newcommand{\epm}{\end{pmatrix}}

\DeclareMathOperator{\ricci}{Ric}

\numberwithin{equation}{section}

\renewcommand{\tilde}{\widetilde}

\renewcommand{\hat}[1]{\widehat{#1}}

\newcommand\ev{\operatorname{even}}
\newcommand\even{\operatorname{ev}}

\newcommand\Ric{\operatorname{Ric}}

\newcommand\Tr{\operatorname{Tr}}

\newcommand\Mas{\text{ as }}

\newcommand\paperintro%
        {%
         }
\newcommand\paperbody%
        {%
         }

\DeclareMathAlphabet{\mathpzc}{OT1}{pzc}{m}{it}

\newcommand{\HH}{\mathbb H}

\newcommand{\RR}{\mathbb R}
\newcommand{\ZZ}{\mathbb Z}
\newcommand{\del}{\partial}

\newcommand{\e}{\epsilon}
\newcommand{\calA}{{\mathcal A}}
\newcommand{\calC}{{\mathcal C}}
\newcommand{\calD}{{\mathcal D}}
\newcommand{\calE}{{\mathcal E}}
\newcommand{\calF}{{\mathcal F}}

\newcommand{\calK}{{\mathcal K}}

\newcommand{\calM}{{\mathcal M}}
\newcommand{\calN}{{\mathcal N}}
\newcommand{\calO}{{\mathcal O}}

\newcommand{\calR}{{\mathcal R}}

\newcommand{\calT}{{\mathcal T}}
\newcommand{\calU}{{\mathcal U}}

\begin{document}
\title{Minimal surfaces with positive genus and finite total curvature
  in $\HH^2\times \RR$} \author{Francisco Mart\'{\i}n \\ University of
  Granada \and Rafe Mazzeo \\Stanford University \and M. Magdalena
  Rodr\'iguez \\ University of Granada}




\maketitle

\begin{abstract}
  We construct the first examples of complete, properly embedded
  minimal surfaces in $\HH^2 \times \RR$ with finite total curvature
  and positive genus.  These are constructed by gluing copies of
  horizontal catenoids or other nondegenerate summands. We also
  establish that every horizontal catenoid is nondegenerate.
\end{abstract}

\section{Introduction}
Amidst the great activity in the past several years concerning the
existence and nature of complete minimal surfaces in homogeneous
three-manifolds, the study of minimal surfaces in $\HH^2 \times \RR$
has witnessed particular success. The central problem is the
solvability of the asymptotic Dirichlet problem, i.e.\ the existence
of complete surfaces asymptotic to a given embedded curve $\gamma$ in
the boundary of the compactification of this space $B^2 \times I$,
where $B^2$ is the closure of the Poincar\'e ball model of $\HH^2$ in
$\RR^2$ and the interval $I$ is the stereographic compactification of
$\RR$.

There have been three main approaches to this problem. The first is
based on the method of Anderson \cite{And} for the analogous problem
in $\HH^3$: one defines a sequence of curves $\gamma_R$ lying on the
geodesic sphere of radius $R$ around some point, solves the Plateau
problem for each of these curves, then attempts to take a limit as $R
\to \infty$. The main points are to show that the sequence of minimal
surfaces with boundary does not drift off to infinity and that the
limit has $\gamma$ as its asymptotic boundary curve; both of these are
accomplished using suitable barrier surfaces, the existence and nature
of which depends upon the convexity of $\HH^3$ at infinity.  This
approach has also been used successfully for the analogous asymptotic
Plateau problem in higher dimensions and codimensions for various
classes of nonpositively curved spaces. The second approach
generalizes the classical method of Jenkins and Serrin~\cite{JS} for
minimal graphs in $\R^3$, and was developed in this setting by Nelli
and Rosenberg~\cite{ner2}, Collin and Rosenberg~\cite{cor2} and Mazet,
Rosenberg and the third author~\cite{marr1}. This involves finding a
minimal graph over domains of $\HH^2$ with prescribed boundary data,
possibly $\pm\infty$.  The third approach is by an analytic gluing
construction, and this is the method we follow here.

Before describing our work, let us draw attention to the issue of
obtaining surfaces with finite total curvature. (We recall that the
total curvature of a surface is defined as the integral on the surface
of its Gauss curvature.)  It turns out to be far easier to obtain
complete minimal surfaces of finite topology in $\HH^2 \times \RR$
with infinite total curvature, and we refer to some of the papers
above for a good (but not yet definitive) existence theory. The
simplest example is the slice $\HH^2 \times \{0\}$, but more generally
there exist minimal surfaces asymptotic to a vertical graph
$\{(\theta, f(\theta)) \colon \theta \in \mathbb{S}^1\} \subset \del
B^2 \times \RR$ for any $f \in \calC^1(\mathbb{S}^1)$. Other examples
include the one-parameter family of Costa-Hoffman-Meeks type surfaces,
each asymptotic to three parallel horizontal copies of $\HH^2$. These
have positive genus and were constructed by Morabito~\cite{mo} also
using a gluing method.  On the other hand, surfaces of finite total
curvature have proved to be more elusive. The basic examples are the
vertical plane $\gamma \times \RR$, where $\gamma$ is a complete
geodesic in $\HH^2$, and the Scherk minimal graphs over ideal polygons
constructed by Nelli and Rosenberg~\cite{ner2}, and Collin and
Rosenberg~\cite{cor2}.  There is also a family of {\it horizontal
  catenoids} $K_\eta$ constructed in~\cite{moro1,pyo1} (called {\it
  2-noids} in~\cite{moro1}), each consisting of a catenoidal handle
which is orthogonal to the vertical direction, and asymptotic to two
disjoint vertical planes which are neither asymptotic nor too widely
separated. The recent paper \cite{hnst} shows that these are the
unique complete minimal surfaces with finite total curvature and two
ends asymptotic to vertical planes. A large number of other examples
of genus zero have been constructed recently by Pyo~\cite{pyo1}, and
Morabito and the third author~\cite{moro1}, independently. Both papers
use the conjugate surface method.  The theory of conjugate minimal
surfaces in $\HH^2 \times \RR$ was elaborated by Daniel~\cite{da2} and
Hauswirth, Sa~Earp and Toubiana~\cite{hst1}. The surfaces
in~\cite{moro1,pyo1} are shown to have total curvature $-4 \pi(k-1)$,
where $k$ is the number of ends, and each end is asymptotic to a
vertical plane. The horizontal catenoids, which have total curvature
$-4\pi$, are a special case.

Despite all this progress, it has remained open whether there exist
complete, properly embedded minimal surfaces in $\HH^2 \times \RR$
with finite total curvature and positive genus.  The aim of this paper
is to construct such surfaces, which we do by gluing together certain
configurations of horizontal catenoids.  There is a dichotomy in the
types of configurations one may glue together.  The ones for which the
horizontal catenoid components have ``necksize'' bounded away from
zero are simpler to handle, and the gluing construction in this case
is quite elementary; the trade-off is that the minimal surfaces
obtained using only this type of component have a very large number of
ends relative to the genus.  Alternatively, one may glue together
horizontal catenoids with very small necksizes, which allows one to
obtain viable configurations with relatively few ends for a given
genus. Unfortunately this turns out to involve more analytic details
because these the horizontal catenoids with very small necks are
`nearly degenerate', and because of this we will address this second
case in a sequel to this paper.

Our main result here is the
\begin{thm}
  For each $g \geq 0$, there is a $k_0 = k_0(g)$ such that if $k \geq
  k_0$, then there exists a properly embedded minimal surface with
  finite total curvature in $\HH^2 \times \RR$, with genus $g$ and $k$
  ends, each asymptotic to a vertical plane.
\label{maingt}
\end{thm}
The proof involves gluing together component minimal surfaces which
are nondegenerate in the sense that they have no decaying Jacobi
fields. Unfortunately, every minimal surface in $\HH^2 \times \RR$
with each end asymptotic to a vertical plane is degenerate since
vertical translation (i.e.\ in the $\RR$ direction) always generates
such a Jacobi field. Because of this we shall work within the class of
surfaces which are symmetric with respect to a fixed horizontal plane
$\HH^2 \times \{0\}$ and then it suffices to work with surfaces which
are {\it horizontally nondegenerate} in the sense that they possess no
decaying Jacobi fields which are even with respect to the reflection
across this horizontal plane. The surfaces obtained in
Theorem~\ref{maingt} are all even with respect to the vertical
reflection, and all are horizontally nondegenerate as well.

This leads to the problem of showing that there are component minimal
surfaces which satisfy this condition, and our second main result
guarantees that many such surfaces exist.
\begin{thm}
  Each horizontal catenoid $K_\eta$ is horizontally nondegenerate.
\label{nondeghc}
\end{thm}

Our final result concerns the deformation theory of this class of
surfaces.
\begin{thm}
  Let $\calM_k$ denote the space of all complete, properly embedded
  minimal surfaces with finite total curvature in $\HH^2\times \RR$
  with $k$ ends, each asymptotic to an entire vertical plane. If
  $\Sigma \in \calM_k$ is horizontally nondegenerate, then the
  component of this moduli space containing $\Sigma$ is a real
  analytic space of dimension $2k$, and $\Sigma$ is a smooth point in
  this moduli space. In any case, even without this nondegeneracy
  assumption, $\calM_k$ is a real analytic space of virtual dimension
  $2k$.
\label{defthy}
\end{thm}
We make two remarks on this. First, this dimension count coincides
with the dimension of the family constructed by our gluing methods,
and also with the dimension of the family of genus $0$ surfaces
constructed in~\cite{moro1}. The fact that the dimension does not
depend on the genus may be surprising at first, but this is also the
case for the space of complete Alexandrov-embedded minimal or CMC
surfaces of finite topology in $\RR^3$, see \cite{KMP} and \cite{PR}.
As is the case in these other theories, it turns out to be very hard
to construct surfaces which are actually degenerate, and we leave this
as an interesting open problem here as well.  The second remark is
that it would also be quite interesting to know whether the vertical
symmetry condition we are imposing is anything more than a technical
convenience.  More specifically, we ask whether there exist finite
total curvature minimal surfaces in $\HH^2 \times \RR$ with vertical
planar ends which do not have a horizontal plane of symmetry.

Our results show that the existence theory for these properly embedded
minimal surfaces of finite total curvature in $\HH^2 \times \RR$ is in
some ways opposite to that in $\RR^3$. Indeed, Meeks, Perez and Ros
\cite{m-p-r} have proved that there is an upper bound, depending only
on the genus, for the number of ends of a properly embedded minimal
surface of finite topology in $\RR^3$. This is a significant step
toward resolving the conjecture of Hoffman and Meeks that a connected
minimal surface of finite topology, genus $g$ and $k>2$ ends can be
properly minimally embedded in $\RR^3$ if and only if $k \leq g + 2$.
By contrast, our result gives some indication that a connected surface
of finite topology and finite total curvature can be properly
minimally embedded in $\HH^2 \times \RR$ only if the number of ends
$k$ has a specific lower bound in terms of the genus $g$. Going out on
a limb, we conjecture that the correct bound for the surfaces
constructed by gluing horizontal catenoids with small necks is $k \geq
2g+1$.  Note also that our construction shows that if there exists a
surface of this type of genus $g$ and $k$ ends, then we can construct
such surfaces with genus $g$ and any larger number of ends, so there
definitely is no upper bound as in the Euclidean space to the number
of ends that a surface of fixed genus may have.

The plan of this paper is as follows: In \S 2 we describe the
horizontal catenoids in more detail, reviewing known properties and
developing some new ones as well. This is where we prove
Theorem~\ref{nondeghc}.  Next in \S 3 we describe the configurations
of approximate minimal surfaces formed by patching together horizontal
catenoids.  The actual gluing, i.e.\ the perturbation of these
approximately minimal surfaces to actual minimal surfaces, which is
possible when some parameter in the construction is sufficiently
large, is carried out in \S 4.  The analytic steps involve a
parametrix construction which is perhaps not so well known in the
minimal surface literature but fairly standard elsewhere; we refer to
the recent paper \cite{MS} which uses a similar method to construct
multi-layer solutions of the Allen-Cahn equation in $\HH^n$.  In \S 5
the general construction is given for gluing together any two
horizontally nondegenerate properly embedded minimal surfaces of
finite total curvature; this is a simple variant of the proof of the
main result. Finally, in \S 6, we study the deformation theory.

The second author is very grateful to the Department of Geometry and
Topology in the University of Granada, where this work was
initiated. R.M. is supported by the NSF grant DMS-1105050. F.M. and
M.M.R. are partially supported by MEC-FEDER Grant no. MTM2011-22547
and a Regional J. Andaluc\'\i a Grant no. P09-FQM-5088.

\section{Horizontal catenoids}
We now describe the fundamental building blocks in our gluing
construction, which are the horizontal catenoids $K_\eta$ in $\HH^2
\times \RR$, originally constructed by Morabito and the third
author~\cite{moro1} and by Pyo~\cite{pyo1}. Each $K_\eta$ has genus
zero and two ends asymptotic to vertical geodesic planes. The
parameter $\eta$ is the hyperbolic distance between these two planes;
it varies in an open interval $(0,\eta_0)$, where the upper bound
$\eta_0$ corresponds to the distance between two opposite sides of an
ideal regular quadrilateral. These catenoids have total curvature
$-4\pi$, and have ``axes'' orthogonal to the $\RR$ direction, whence
the moniker horizontal.

\medskip
\noindent{\bf The horizontal catenoid as a vertical bigraph:}
The initial construction of $K_\eta$ in the papers above describes it
as a bigraph over a region $\Omega_\eta \subset \HH^2$ with a
reflection symmetry across the central $\HH^2 \times \{0\}$. This
means the following: first, there is a nonnegative function $u$
defined in $\Omega_\eta$ such that
\[
K_\eta = \{ (z, u(z)): z \in \Omega_\eta\} \cup \{(z,-u(z)): z \in
\Omega_\eta\}.
\]
The domain $\Omega_\eta$ is bounded by four smooth curves of infinite
length which intersect only at infinity; two of these are
hyperparallel geodesics, denoted $\gamma_{-1}$ and $\gamma_1$, and the
parameter $\eta$ equals the hyperbolic distance between them; the
other two curves, denoted $C_{-1}$ and $C_1$, connect the adjacent
pairs of endpoints of $\gamma_{\pm 1}$. The function $u$ is strictly
positive in the interior of $K_\eta$, vanishes and has infinite
gradient on $C_{-1} \cup C_1$, and tends to $+\infty$ along
$\gamma_{-1} \cup \gamma_1$.  We also let $C_{-1}'$ and $C_1'$ be the
geodesic lines with the same endpoints as $C_{-1}$ and $C_1$,
respectively, and $\Omega_{\eta}'$ the ideal geodesic quadrilateral
bounded by $\gamma_{-1} \cup \gamma_1 \cup C_{-1}' \cup C_1'$.  Using
vertical planes (which are minimal) as barriers, we see that $C_{-1}$
and $C_1$ are strictly concave with respect to $\Omega_\eta$. In
particular, they lie in the interior of $\Omega_\eta'$.  For later
reference, we identify a few other curves which will enter the
discussion. First, let $\Gamma$ denote the unique geodesic which is
orthogonal to both $\gamma_{\pm 1}$; next, let $\gamma_0$ be the
geodesic perpendicular to $\Gamma$ and midway between $\gamma_{\pm
  1}$; finally, denote by $\tilde{\gamma}_{\pm 1}$ the two geodesics
which connect the opposite ideal vertices of $\Omega_\eta$.  Observe
that $\gamma_0$ is perpendicular to both $C_{\pm 1}'$; in addition the
points of intersection $\gamma_0 \cap \Gamma$ and $\tilde{\gamma}_{-1}
\cap \tilde{\gamma}_1$ are the same, and we denote this centerpoint by
$Q$.

We finish this discussion by noting that the horizontal catenoid with ends asymptotic 
to the two vertical planes $\gamma_1\times\R$ and $\gamma_{-1}\times\R$ is unique
(when it exists). This follows from the fact that this surface is a bigraph across the plane $t=0$
as well as in the two horizontal directions associated to the geodesics $\Gamma$ and 
$\gamma_0$ (see Proposition~\ref{horbigraph}).
\begin{figure}[htbp]
    \begin{center}
        \includegraphics[width=.66\textwidth]{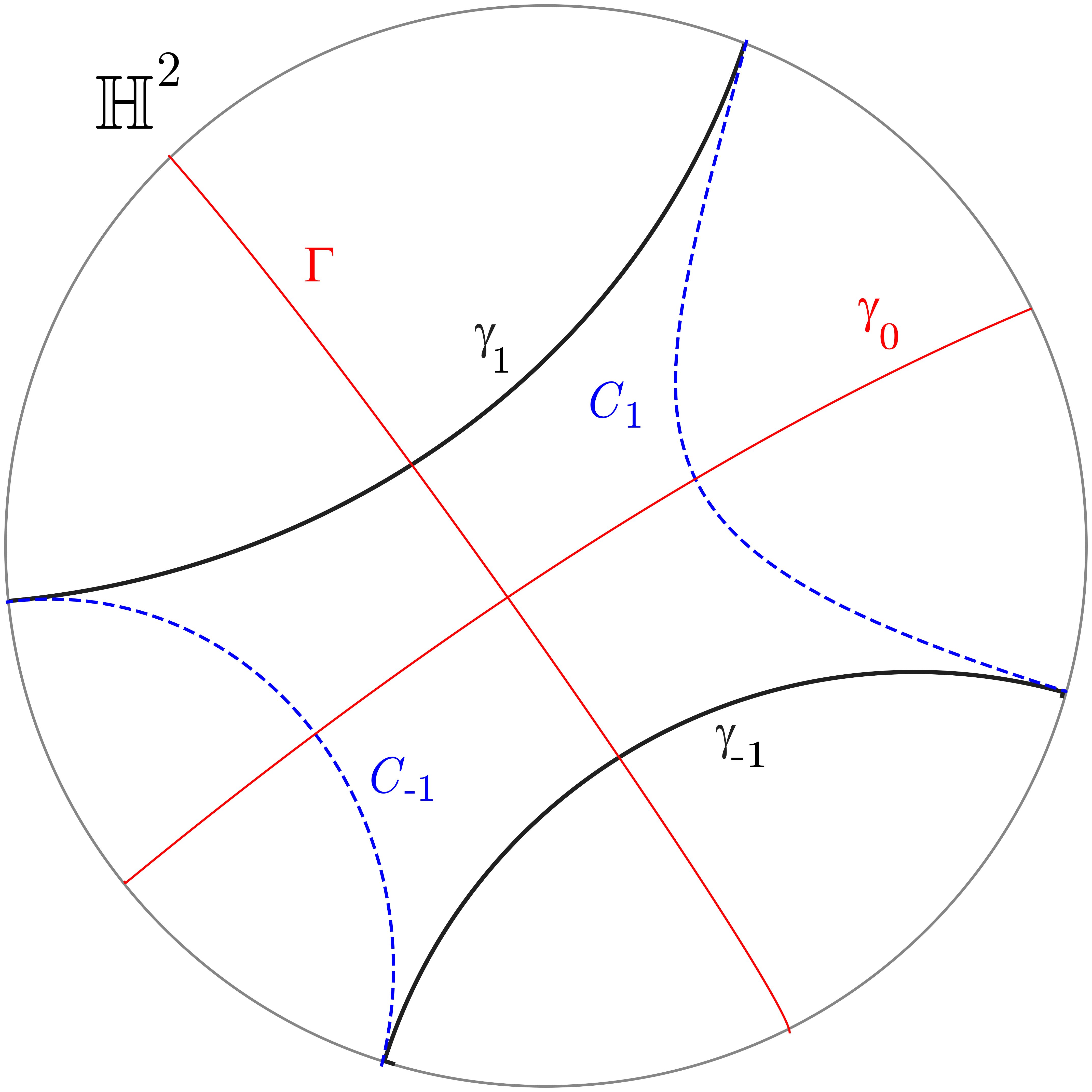}
     \end{center}
   \caption{The boundary of the region $\Omega_\eta.$} \label{fig:one}
\end{figure}
\medskip

\noindent{\bf The extremal surface:}
The family of catenoids $K_\eta$ exists only for $0 < \eta <
\eta_0$. This critical value $\eta_0$ corresponds to the case where
the pairs of geodesics $\tilde{\gamma}_{\pm 1}$ intersect orthogonally
at $Q$. The limiting domain $\Omega_{\eta_0}$ is the same as
$\Omega_{\eta_0}'$ (so $C_{-1} = C_{-1}'$ and $C_1 = C_1'$ in this
limit).  Furthermore, as $\eta \nearrow \eta_0$, the value $u(Q)$
tends to $+\infty$.  In fact, recentering $K_\eta$ by translating down
by $-u(Q)$, there is a limiting surface which is a graph over
$\Omega_{\eta_0}'$ taking the boundary values $\pm\infty$ on
alternate sides. It is planar of genus zero with one end.  This surface is
qualitatively similar to the classical Scherk surface of $\R^3$, and
so we also call it the Scherk surface. As already mentioned in the introduction,
this example was constructed in \cite{cor2, ner2}.

\medskip

\noindent{\bf Further symmetries:} Unlike the Euclidean case, or even
the case of vertical catenoids in $\HH^2 \times \RR$, the horizontal
catenoid $K_\eta$ has only a discrete isometry group, isomorphic to
$\ZZ_2 \times \ZZ_2 \times \ZZ_2$. Each $\ZZ_2$ corresponds to a
reflection: the first reflection, which we call ${\cal R}_t$, sends
$(z,t)$ to $(z,-t)$, and thus interchanges the top and bottom halves
of $K_\eta$; the second, ${\cal R}_s$, is the reflection across $\Gamma
\times\R$, it interchanges the `left' and `right' sides of each
asymptotic end; the final one, ${\cal R}_o$, is the reflection across
$\gamma_0\times\R$ and interchanges the two ends of $K_\eta$ and has
fixed point set a loop around the neck.
\begin{figure}[htbp]
    \begin{center}
        \includegraphics[height=.55 \textheight]{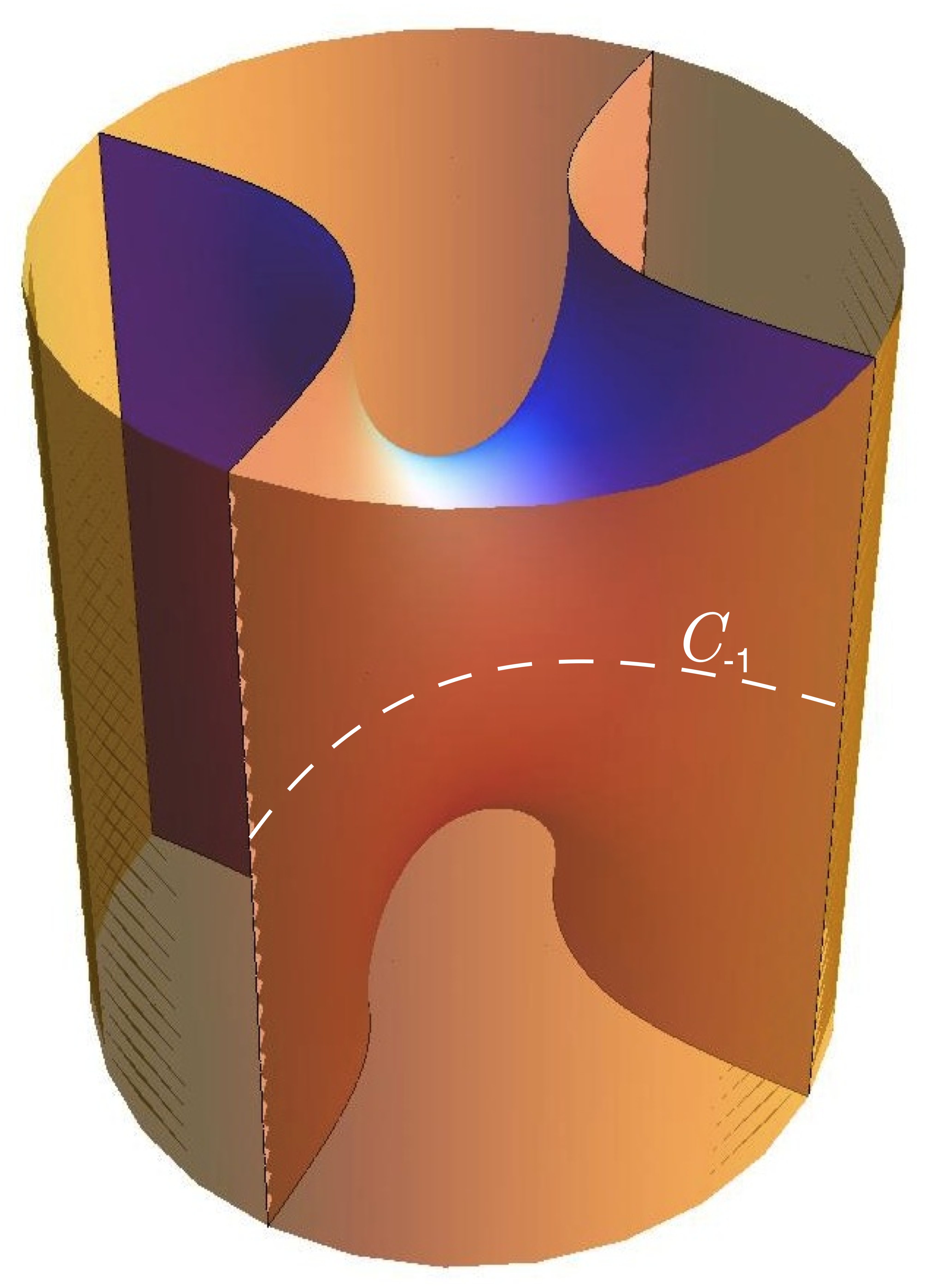}
     \end{center}
   \caption{A horizontal catenoid  $K_\eta.$} \label{fig:intro}
\end{figure}
\medskip

\noindent{\bf Graphical representation of the ends of $K_\eta$:} 
Each end of $K_\eta$ is asymptotic to one of the totally geodesic
vertical planes $P_j = \gamma_j \times \RR$, $j=\pm 1$.  The
intermediate vertical plane $\gamma_0\times\R$ fixed by ${\cal R}_o$ bisects
$K_\eta$, decomposing it into two pieces, $K_\eta^{1} \cup
K_\eta^{-1}$, which are interchanged by this reflection.  Each
$K_\eta^j$ is a smooth manifold with compact boundary and one end,
which is asymptotic to the vertical plane $P_j$.  Outside of some
compact set, $K_\eta^j$ is a normal graph over $P_j$, with graph
function $v^j$ which is strictly positive and defined on an exterior
region $E_\eta^j = P_j \setminus \calO_\eta$.

The two ends are equivalent, so let us fix one and drop the sub- and
superscripts $j$ for the time being.  Use parameters $(s,t)$ on $P$,
where $t$ is the vertical coordinate and $s$ is the signed distance
function along the geodesic $\gamma$, as measured from $\gamma \cap
\Gamma$. The restrictions of ${\cal R}_s$ and ${\cal R}_t$ to the
plane $P$ correspond to $(s,t) \mapsto (-s,t)$ and $(s,t) \mapsto
(s,-t)$, respectively.  We assume that the domain $E_{\eta}$ is
invariant under both these reflections.

The parameter $\eta$, strictly speaking, measures the distance between
the asymptotic vertical planes of $K_\eta$, but also measures the size
of the neck of $K_\eta$, which we take, for example, as the length of
the closed curve $K_\eta\cap(\gamma_0\times\R)$. This function, which
we denote by $n(\eta)$, has $n(\eta) \to 0$ as $\eta \to 0$ and
$n(\eta) \to \infty$ as $\eta \to \eta_0$. This can be thought as the
original parameter for this family used in~\cite{moro1, pyo1}.

We now describe the asymptotic decay profile of the graphical
representation of $K_\eta$ over $P$.  Introduce polar coordinates
$(r,\theta)$ in the $(s,t)$ plane, so $s = r\cos\theta$ and $t = r
\sin \theta$, where the coordinates $(s,t)$ have been defined above.
\begin{prop}
  For each $\eta \in (0, \eta_0)$, as $r \to \infty$, the graph
  function $v$ has the asymptotic expansion
\begin{equation}
  v(r,\theta) = A_\eta(\theta) r^{-\frac12} e^{-r} + \calO(r^{-\frac32} e^{-r}),
  \label{decayv}
\end{equation}
where $A_\eta(\theta)$ is some strictly positive smooth function on
$\mathbb{S}^1$.
\label{decayvprop}
\end{prop}
This decay profile is essentially a linear phenomenon and corresponds
to the known asymptotic properties of homogeneous solutions of the
Jacobi operator on $P$.  Recall that for any minimal surface $\Sigma$,
its Jacobi operator (for the minimal surface equation) is the elliptic
operator
\begin{equation}
  L_{\Sigma} :=  \Delta_{\Sigma} + |A_{\Sigma}|^2 + \ricci\,(N,N);
  \label{JacobiopSigma}
\end{equation}
here $\Delta_{\Sigma}$ is the Laplacian on $\Sigma$, $A_\Sigma$ the
second fundamental form of the surface, $N$ its unit normal, and
$\ricci$ the Ricci tensor of the ambient space.  When $\Sigma = P$ is
a vertical plane, this simplifies substantially. Indeed, $A_P \equiv
0$ and $N$ has no vertical component, so that $\ricci(N, N) \equiv
-1$, hence
\begin{equation}
  L_P = \Delta_{\RR^2} - 1.
  \label{JacobiopP}
\end{equation}

We now deduce Proposition~\ref{decayvprop} from a slightly more
general result.
\begin{prop}
  Let $E \subset P$ be an unbounded region with complement $P
  \setminus E$ smoothly bounded and compact. Let $K \subset \HH^2
  \times \RR$ be a minimal surface which is a normal graph over $E$
  with compact boundary over $\del E$, and denote by $v: E \to \RR$
  the graph function. Suppose that $v \to 0$ at infinity in $P$. Then
  there exists $A \in \calC^\infty(\mathbb{S}^1)$, such that
  \begin{equation}
    v(r,\theta) = A(\theta) r^{-\frac12}e^{-r} + \calO(r^{-\frac32} e^{-r}).
    \label{genasym}
  \end{equation}
  Furthermore, if $K$ lies on one side of $P$ at infinity, then $A$ is
  either strictly positive or strictly negative.
\label{genasymprop}
\end{prop}
\begin{proof}
  The minimal surface equation for a horizontal graph over $P$ is a
  quasilinear elliptic equation $\calN(v,\nabla v,\nabla^2 v)=0$, the
  linearization of which at $v=0$ is just $L_P$. Let $p_j$ be any
  sequence of points in $P$ tending to infinity, and consider the
  restriction of $v$ to the unit ball $B_1(p_j)$ around
  $p_j$. Recenter this ball at the origin and write the translated
  function as $v_j$.  We are assuming that $v_j \to 0$, and it follows
  from standard regularity theory for the minimal surface equation
  that
  \begin{equation}
    ||v_j||_{2,\mu; B_1(0)} \to 0\quad \Mas\quad j \to \infty,
    \label{locest}
  \end{equation}
  where $||\cdot ||_{2,\mu; B_1(0)}$ denotes the norm on the Holder
  space ${\cal C}^{2,\mu}$ on the unit ball $B_1(0)$ (see~\cite{gt}).
  This means that we can write
  \begin{equation}
    \calN(v, \nabla v, \nabla^2 v) = L_P v + Q(v),
    \label{Taylor}
  \end{equation}
  where $Q$ is quadratic in $v$, $\nabla v$ and $\nabla^2 v$, and has
  the property that if $||v||_{2,\mu}$ is small, then
  \begin{equation}
    ||Q(v)||_{0,\mu} \leq C ||v||_{2,\mu}^2.
    \label{quadrem}
  \end{equation}

Now, applying the inverse $G_P = (\Delta_{\RR^2} - 1)^{-1}$ of the Jacobi operator 
(this $G_P$ is also called the Green operator or Green function) to \eqref{Taylor}
shows that $\calN(v,\nabla v, \nabla^2 v) = 0$ is equivalent to the equation
  \begin{equation}
    v = G_P ( -Q(v)).
    \label{Green1}
  \end{equation}

  We assume initially only that $v \to 0$ at infinity in $P$, but
  without any particular rate.  We first show that $v$ decays at some
  exponential rate; this is done using the maximum principle. The
  second and final step is to obtain the asymptotic formula
  \eqref{genasym}.

  To begin, using \eqref{locest} and \eqref{quadrem}, the following is
  true: There exists a constant $C_1 > 0$ such that, given any
  $\delta_0 > 0$ sufficiently small, there exists $R_0 \geq 1$ so that
  if $\delta < \delta_0$, $R > R_0$ and $|v| < \delta$ for all $r \geq
  R$ then $\sup |Q(v)| \leq C_1 \delta^2$.

  Now, define $w = a e^{-r} + b$. This satisfies $L_P w = -a r^{-1}
  e^{-r} - b$. Suppose that $\delta<\delta_0,1$ and $R>R_0$ are such
  that $\sup_{r\geq R}|v| =\delta$ is attained at $r = R$, and choose
  the coefficients $a$ and $b$ so that $a e^{-R} + b \geq \delta$ and
  $b \geq C_1 \delta^2$; to be specific, we take $b = C_1 \delta^2$
  and $a = \delta( 1 - C_1 \delta) e^R$. Then $v - w \leq 0$ when $r =
  R$, and furthermore (taking $\delta\leq 1/C_1$),
  \[
  L_P( v - w) = - Q(v) + ar^{-1}e^{-r} + b \geq -Q(v) + C_1 \delta^2
  \geq 0,
  \]
  where we drop the middle term since $a r^{-1}e^{-r} > 0$. Thus $v-w$
  is a subsolution of the equation which is non-positive at $r = R$
  and is bounded as $r \geq R$, hence $v - w \leq 0$ for all $r \geq
  R$. This implies that
  \[
  v(R+1, \theta) \leq w( R+1) = \delta (1 - C_1 \delta) e^{R} e^{-R-1}
  + C_1 \delta^2 = \delta \left( ( 1 - C_1\delta) e^{-1} + C_1\delta
  \right).
  \]
  Since $C_1$ is independent of $\delta$, we can choose $\delta$ so
  small that $(1 - C_1 \delta) e^{-1} + C_1 \delta < \frac12$, and
  hence $v(R+1,\theta) \leq \frac12 \delta$.  In other words, we see
  that
  \[
  \sup_{r = R+1} |v| \leq \frac12 \sup_{r = R} |v|,
  \]
  for all $R \geq R_0$, or equivalently $|v(r,\theta)|\leq C e^{-m r}$
  for some $m > 0$. This completes the first step.

  Now, by local a priori estimates, if $\mathcal{A}(\rho)$ is the
  annulus $\{\rho \leq r \leq \rho+1\}$, then
  $||v||_{2,\mu;\mathcal{A}(\rho)} \leq C e^{-m \rho}$, and hence
  $|Q(v)| \leq C_2 e^{-2m r}$ for all $r \geq R_0$.  Assuming that $m
  < 1$, we use the maximum principle again, this time with $w =
  e^{-\beta r}$ for some $\beta \in (m, \min\{1, 2m\})$. Since
  \[
  L_P w = (\beta^2 - r^{-1} \beta -1) e^{-\beta r} < (\beta^2-1)
  e^{-\beta r},
  \]
  we obtain $L_P(v - C_3 w) \geq -Q(v) + C_3 (1-\beta^2)e^{-\beta r}
  \geq 0$ for all $r \geq R_0$; in addition $v - C_3 w \leq 0$ along
  $r = \rho$ for $C_3$ sufficiently large, and $v - C_3 w \to 0$ as $r
  \to \infty$.  We conclude that $v \leq C_3 e^{-\beta r}$ for $r \geq
  R_0$.

With this argument we have improved the exponent in the decay rate
from $m < 1$ to any $\beta \in (m, \min\{1, 2m\})$. Iterating this a finite
number of times shows that we can obtain a decay rate with exponent
as close to $1$ as we please. In other words, we conclude that $v \leq C_4
e^{-(1-\e)r}$ for some very small $\e > 0$, and hence $|Q(v)| \leq
C_5 e^{-2(1-\e) r}$, and then that $||Q(v)||_{0,\mu;\calA(\rho)}
\leq C_6 e^{-2(1-\e)\rho}$ as well.

  Now write $v = G_P( -Q(v))$ as in \eqref{Green1}. Since $L_P$ commutes with rotations
and translations in $P$, the Green function $G_P( (s,t), (s', t'))$ depends only on the (Euclidean)
distance between $(s,t)$ and $(s',t')$, and hence reduces to a function of one variable which
satisfies a modified Bessel equation. We thus arrive at the well-known classical formula 
\begin{equation}
G_P( (s,t), (s',t')) = \frac{1}{4\pi} K_0 ( \sqrt{ |s-s'|^2 + |t-t'|^2}).
\label{Green2}
\end{equation}
Here $K_0(r)$ is the Bessel function of imaginary argument, see
  \cite{Lebedev}, which has the well-known asymptotics
  \begin{equation}
    \begin{aligned}
      K_0(r) & \sim \log r\ \mbox{as}\ r \searrow 0,  \\
      K_0(r) & \sim r^{-\frac12} e^{-r} + \calO( r^{-\frac32}e^{-r})\
      \mbox{as}\ r \nearrow \infty
    \end{aligned}
    \label{Green3}
  \end{equation}
  (we are omitting the normalizing constant $(4\pi)^{-1}$ for
  simplicity.)  It is a straightforward exercise to check that if $f$
  is continuous and $|f| \leq C e^{-2(1-\e)r}$, then
  \begin{multline*}
    v = G_P f  = \int_{\RR^2} G_P( (s,t), (s',t')) f(s',t') \, ds' dt'  \\
    = A(\theta) r^{-1/2} e^{-r} + \calO(r^{-3/2}e^{-r}),
  \end{multline*}
  and if $f \in \calC^\infty$, then $A \in
  \calC^\infty(\mathbb{S}^1)$. We refer to \cite{Melrose} for an
  explanation of the linear mapping $f(s,t) \mapsto A(\theta)$ (it is
  the adjoint of the Poisson operator and is closely related to the
  scattering operator for $L_P$).

  To complete the argument, suppose that $v > 0$. Since
  $A(\theta)e^{-r}r^{-1/2}$ dominates the expansion for $r$ large,
  clearly $A(\theta) > 0$.
\end{proof}

\medskip

\noindent{\bf Asymptotics of Jacobi fields:} Let $\Sigma$ be a
  complete properly embedded minimal surface in $\HH^2 \times \RR$
  with a finite number of ends, each one asymptotic to a vertical
  plane $P_\alpha$, $\alpha \in A$. We could also let $\Sigma$ be an
  exterior region in any such surface, i.e. the discussion below
  incorporates the case where $\Sigma$ has compact boundary. We now
  recall some facts about the asymptotic properties of solutions of
  the equation $L_\Sigma \psi = 0$, where $L_\Sigma$ is the Jacobi
  operator \eqref{JacobiopSigma}. This operator has the particularly
  simple form \eqref{JacobiopP} when $\Sigma$ is a vertical plane $P$,
  and this provides the asymptotic model for $L_\Sigma$ in our more
  general setting.  When $\Sigma$ is a horizontal catenoid $K_\eta$,
  we write the Jacobi operator as $L_\eta$. 

There are many classical sources for the material in this
  section; we refer in particular to \cite{Melrose} since the
  treatment is specifically geometric. 

It is a classical fact in scattering theory that any solution
  of $L_P \psi = 0$ (defined either on all of $P$ or just on the
  complement of a relatively compact domain) has a so-called far-field
  expansion as $r \to \infty$; this takes the form
  \begin{equation}
    \psi(r,\theta) \sim  (F^+(\theta) r^{-\frac12} + \calO(r^{-\frac 32})) e^r + (F^-(\theta) r^{-\frac12} + \calO(r^{-\frac32})) e^{-r}.
    \label{farfield}
  \end{equation}
  In the particular case where $P$ is all of $\RR^2$ and has no
  boundary, then $F^-(\theta) = F^+(-\theta)$, but in general the
  relationship is more complicated. The subtlety in such an expansion
  is that the coefficients $F^\pm(\theta)$ are allowed to be arbitrary
  distributions on $\mathbb{S}^1$, and if these coefficients are not
  smooth, then the expansion must be interpreted weakly, i.e.\ as
  holding only after we pair with an arbitrary test function
  $\varphi(\theta)$.  The simplest `plane wave' solution of this
  equation, $e^s$, exhibits an expansion with coefficients which are
  Dirac delta functions:
  \[
  e^s \sim \delta(\theta) r^{-\frac12} e^r + \delta( -\theta)
  r^{-\frac12} e^{-r}.
  \]
  One can interpret this as reflecting the obvious fact that this
  solution grows exponentially as $s \to \infty$ (which corresponds to
  $\theta = 0$) and decays exponentially as $s \to -\infty$ (which is
  $\theta = \pi$).  On the other hand, the Green function for this
  operator with pole at $0$, $G_P( (s,t), (0,0)) = K_0(r)$, has
  \[
  G_P \sim r^{-\frac12} e^{-r}\ \ \mbox{as}\ r \to \infty,
  \]
  i.e.\ $F^+ = 0$ and $F^- = 1$. Similarly, since $\del_s$ commutes
  with $L_P$, the function $\del_s G_P$ is another Jacobi field, and
  it has
  \[
  \del_s G_P \sim r^{-\frac12} \cos \theta \, e^{-r}.
  \]

Now return to the Jacobi operator on more general minimal
  surfaces with ends asymptotic to vertical planes.
 \begin{prop}
 Suppose that $L_\Sigma \psi = 0$. Let $r_\alpha$ denote the radial function on the asymptotic end $P_\alpha$,
 and transfer this (via the horizontal graph description) to a function on $\Sigma$. Then $\psi$ has the far-field expansion
 \[
 \psi \sim \sum_{\alpha \in A} (F^+_\alpha(\theta) r_\alpha^{-\frac12}
 + \calO(r_\alpha^{-\frac 32})) e^{r_\alpha} + (F^-_\alpha(\theta)
 r_\alpha^{-\frac12} + \calO(r_\alpha^{-\frac32})) e^{-r_\alpha}.
 \]
 \label{asympprop}
\end{prop}
The set of possible leading (distributional)
coefficients $\{ F^+_\alpha, F^-_\alpha\}$ which can occur is called
the scattering relation for $L_\Sigma$. If $\Sigma$ is preserved by
the reflection ${\cal R}_t$ and we restrict to functions which are
even with respect to ${\cal R}_t$, then any collection
$\{F^+_\alpha\}$ uniquely determines a solution, and hence determines
the other set of coefficients $\{F^-_\alpha\}$; the same is true on
the complement of a finite dimensional subspace if we drop the
evenness condition.  The map $\{F^+_\alpha\} \mapsto \{F^-_\alpha\}$
is called the scattering operator.


 \medskip

\noindent{\bf Geometric Jacobi fields:}
We now describe the special family of global Jacobi fields on the
horizontal catenoid $K_\eta$ generated by the `integrable', or
geometric, deformations of $K_\eta$. In other words, these Jacobi
fields are tangent at $K_\eta$ to families of horizontal catenoids.

We have already described the space $\calC_K$ of all horizontal
catenoids which are symmetric about the plane $t=0$. Indeed, there is
a unique such catenoid associated to any two geodesics $\gamma_{\pm}$
in $\HH^2$ with $0<\mbox{dist}\,(\gamma_+, \gamma_-) = \eta <
\eta_0$. Thus $\calC_K$ is identified with an open subset of the space
of distinct four-tuples of points on $\mathbb{S}^1$: writing any such
four-tuple in consecutive order around $\mathbb{S}^1$ as
$(\zeta_{-,1}, \zeta_{-,2}, \zeta_{+,1}, \zeta_{+,2})$, then we let
$\gamma_\pm$ be the unique geodesic connecting $\zeta_{\pm,1}$ to
$\zeta_{\pm,2}$.  Note that we do not allow arbitrary four-tuples
simply because the distances between these geodesics must be less than
$\eta_0$.  In any case, $\dim \calC_K = 4$.
 
There are various different ways to describe the complete family of
horizontal catenoids (symmetric about $\{t=0\}$). First we can vary
the points $\zeta_{\pm,\ell}$ independently. Second, we can transform
$K_\eta$ using the three-dimensional space of isometries of $\HH^2$,
and then, to obtain the entire four-dimensional family, we augment
this by the extra deformation corresponding to changing the parameter
$\eta$, i.e.\ moving the geodesics relative to one another.

Using the first parametrization of this family, let $\zeta(\epsilon)$
be a smooth curve in the space of (allowable) four-tuples where we
vary only one end of one of the geodesics. The corresponding Jacobi
field decays exponentially in all directions but one (this holds by
Proposition~\ref{decayvprop} and the behavior of the hyperbolic
metric at infinity). For example, if we vary only $\zeta_{+,2}$, then
this Jacobi field decays exponentially in all directions at infinity
on $P_-$, while on $P_+$, it decays exponentially as $s \to -\infty$
but grows exponentially as $s \to +\infty$ (we assume that $s$
increases as we move along $\gamma_+$ from $\zeta_{+,1}$ to
$\zeta_{+,2}$).

In computing the infinitesimal variations here, note that if
$K_\eta(\epsilon)$ is a one-parameter family of horizontal catenoids
as described here, with $K_\eta(0) = K_\eta$, then for $\epsilon \neq
0$ we can write $K_\eta(\epsilon)$ as a normal graph over some proper
subset of $K_\eta$. However, as $\epsilon \to 0$, this proper subset
fills out all of $K_\eta$, and hence the derivative of the normal
graph function at $\epsilon = 0$ is defined on the entire surface.

Denote by $\Phi_{\pm, \ell}$ the Jacobi field generated by varying only the one 
point $\zeta_{\pm, \ell}$, and note that each $\Phi_{\pm,  \ell} \sim e^s = e^{r\cos\theta}$.  
For any four real numbers $E_{\pm,   \ell}$, $\ell = 1,2$, we define
\[
\Phi_E = \sum_{\pm, \ell} E_{\pm,\ell} \Phi_{\pm,\ell},
\]

\medskip

\noindent{\bf $K_\eta$ as a horizontal bigraph:}
The geometric Jacobi fields can be used to show that $K_\eta$ is a
horizontal bigraph in two distinct directions: over the vertical plane
$\gamma_0 \times \RR$ and also over the vertical plane $\Gamma \times
\RR$.  These two new graphical representations were also obtained in
the recent paper \cite{hnst} using an Alexandrov reflection
argument. We present a separate argument using these Jacobi fields
since it is somewhat less technical.  Note that the assertion about
horizontal graphicality must be clarified first since there are two
geometrically natural ways of writing a surface with a vertical end in
$\HH^2 \times \RR$ as a horizontal graph over a vertical plane.
Indeed, let $\gamma(s)$ be an arclength parametrized geodesic in
$\HH^2$. We can then coordinatize $\HH^2$ using Fermi coordinates off
of $\gamma$, i.e.\ $(s,\sigma) \mapsto \exp_{\gamma(s)}(\sigma
\nu(s))$ (where $\nu$ is the unit normal), or else by $(s,\sigma)
\mapsto D_\sigma(\gamma(s))$, where $D_\sigma$ is the one-parameter
family of isometries of $\HH^2$ which are dilations along the geodesic
$\gamma^\perp$ orthogonal to $\gamma$ and meeting $\gamma$ at
$\gamma(0)$.  We use the latter, and then say that a curve is a graph
over $\gamma$ in the direction of $\gamma^\perp$ if $\sigma = f(s)$.
Hence $f \equiv \mbox{const.}$ corresponds to a geodesic $\gamma'$
which is hyperparallel to $\gamma$ and perpendicular to
$\gamma^\perp$.  This transfers immediately to the notion of a
horizontal graph over $\gamma \times \RR$ in the direction of
$\gamma^\perp$ in $\HH^2 \times \RR$.

Now, recall the two orthogonal geodesics $\Gamma$ and $\gamma_0$ (see
Figure \ref{fig:one}). The vertical plane $\Gamma\times\R$
(resp. $\gamma_0\times\R$) fixed by ${\cal R}_s$ (resp. ${\cal R}_o$) bisects
$K_\eta$, decomposing it into two pieces denoted by $K_\eta^{s,1},
K_\eta^{s,-1}$ (resp. $K_\eta^{o,1}, K_\eta^{o,-1}$), which are
interchanged by this reflection.  The result of Hauswirth, Nelli, Sa
Earp and Toubiana \cite[Lemmas 3.1 and 3.2]{hnst} is the following:
\begin{prop}
  For each $\eta \in (0,\eta_0)$, 
  $K_\eta^{o,+}$ is a horizontal graph in the
  direction of $\Gamma$ over some portion of the vertical plane
  $\gamma_0 \times \RR$ while $K_\eta^{s,+}$ is a horizontal graph in
  the direction of $\gamma_0$ over some portion of the vertical plane
  $\Gamma \times \RR$.
  \label{horbigraph}
\end{prop}
As noted, we sketch an independent proof of this.
\begin{proof}
  First, notice that the first assertion in Proposition
  \ref{horbigraph} is equivalent to the fact that the Jacobi field
  $\Phi_o$ generated by dilations along $\Gamma$ is strictly positive
  on $K_\eta^{o,+}$ and vanishes along the fixed point set of ${\cal R}_o$;
  similarly, the second assertion is equivalent to claiming that the
  Jacobi field $\Phi_s$ generated by dilations along $\gamma_0$ is
  strictly positive on $K_\eta^{s,+}$ and vanishes along the fixed
  point set of ${\cal R}_s$.

  The proof has two steps. We first show that these Jacobi fields have
  the required positivity property when $\eta$ is very close to the
  upper limit $\eta_0$. We then show that as we vary $\eta$ from
  $\eta_0$ down to $0$, they maintain this positivity.

  We begin by asserting that the limiting Scherk surface $K_{\eta_0}$
  is a horizontal bigraph over $\Gamma \times \RR$ and also over
  $\gamma_0 \times \RR$. In fact, this surface has a symmetry obtained
  by rotating by $\pi/2$ and flipping $t \mapsto -t$; this
  interchanges these two graphical representations.  This can be
  proved by a simple Alexandrov reflection argument: Consider the
  family of geodesics $\gamma_\sigma$ perpendicular to $\Gamma$ and
  intersecting it at $\Gamma(\sigma)$ (where $\Gamma(0) = \Gamma \cap
  \gamma_0$). The plane $\gamma_\sigma \times \RR$ only intersects
  $K_{\eta_0}$ when $\sigma < \eta_0/2$, and for $\sigma$ just
  slightly smaller, the reflection of the `smaller' portion of
  $K_{\eta_0}$ across this vertical plane does not intersect the other
  component. Pushing $\sigma$ lower, it is standard to see that these
  two half-surfaces do not intersect until $\sigma = 0$, in which case
  they coincide.  These planes of reflection are the images of
  $\gamma_0 \times \RR$ with respect to dilation along $\Gamma$, so we
  deduce that the vector field $X$ generated by this dilation is
  everywhere transverse to the component $K_{\eta_0}^+$ of
  $K_{\eta_0}$ on one side of this plane of symmetry. Note finally
  that the angle between $X$ and $K_{\eta_0}^+$ is bounded below by a
  positive constant if we remain a bounded distance away from
  $\partial K_{\eta_0}^+$.

  Now recall that an appropriate vertical translate of $K_\eta$
  converges locally uniformly in $\calC^\infty$ to $K_{\eta_0}$, and
  indeed this convergence (of the translated $K_\eta$) is uniform in
  the half-plane $t \geq -C$ for any fixed $C$. It is then clear that
  the angle between $X$ and $K_\eta^{o,+}\cap \{t\geq -C\}$ is also
  positive everywhere when $\eta$ is sufficiently close to $\eta_0$.
  Since $K_\eta$ is invariant by ${\cal R}_t$, this finishes the first
  step.

  For the second step, to be definite consider $\Phi_o$, and let us
  study what happens as $\eta$ varies in the interval $(0,\eta_0)$. We
  use that $L_\eta \Phi_o = 0$ and $\Phi_o$ is nonnegative on
  $K_\eta^{o,+}$ for $\eta$ close to $\eta_0$, vanishing only on the
  boundary, and by the Hopf boundary point lemma, has strictly
  positive normal derivative there.  As $\eta$ decreases, $\Phi_o$
  must remain strictly positive in the interior; the alternative would
  be that it develops some interior zeroes or else its normal
  derivative vanishes at the boundary while the function still remains
  nonnegative in the interior, and both contradict the maximum
  principle. Note that we are using two additional facts: first, we
  use the form of the maximum principle which states that a
  nonnegative solution of $(\Delta + V) u = 0$ cannot have an interior
  zero, regardless of the sign of $V$; {we also use that because of
    the graphical representation of the ends, it is clear that
    $\Phi_o$ is bounded away from $0$ outside a compact set.}  This
  proves that $\Phi_o > 0$ on $K_\eta^{o,+}$ for all $\eta \in
  (0,\eta_0)$, which shows that this half remains graphical.

  The case of the Jacobi field $\Phi_s$ is quite similar. Taking into
  account the asymptotic behavior of $K_\eta$, it is not hard to see
  that there exists a constant $T \gg 0$ so that $K_\eta \cap \{
  |t|>T\}$ is a horizontal graph over the vertical plane $\Gamma
  \times \RR$, $\forall \eta \in (0,\eta_0)$. We can then apply the
  same argument as in the previous paragraphs to $K_\eta \cap \{ |t|
  \leq T\}$.
\end{proof}

\medskip

\noindent{\bf Fluxes:}
Closely related to the geometry in the last subsection is the
computation of the flux homomorphism.  We recall that if $\Sigma$ is
an oriented minimal surface in an ambient space $(Z,g)$, then its flux
is a linear mapping
\[
\calF: H_1(\Sigma) \times \calK(Z,g) \longrightarrow \RR,
\]
where $\calK(Z,g)$ is the space of Killing vector fields on $Z$, i.e.\
infinitesimal generators of one-parameter families of isometries. The
definition is simple: if $c \in H_1(\Sigma)$ is a homology class
represented by a smooth oriented closed curve $\gamma$ and if $X \in
\calK(Z,g)$, then
\[
\calF(c, X) = \int_\gamma X \cdot \nu\, ds,
\]
where $\nu$ is the unit normal to $\gamma$ in $\Sigma$.  This is only
interesting when the ambient space $Z$ admits Killing fields, but this
is certainly the case in our setting. Indeed, $\calK( \HH^2 \times
\RR)$ (with the product metric) is four-dimensional: there is one
Killing field $X_t$ generated by vertical translation, and a
three-dimensional space of Jacobi fields on $\HH^2$ which lift to the
product to act trivially on the $\RR$ factor.  If $K_\eta$ is a
horizontal catenoid and if $o = \gamma_0 \cap \Gamma \in \HH^2$ is its
`center', then this three-dimensional space is generated by the
infinitesimal rotation $X_R$ around $o$, and the infinitesimal
dilations $X_{\gamma_0}$ and $X_\Gamma$ along $\gamma_0$ and $\Gamma$,
respectively.

The first homology (with real coefficients), $H_1(K_\eta)$, is
one-dimensional and is generated by the loop $(\gamma_0 \times \RR)
\cap K_\eta$. Thus it suffices to consider $\calF([\gamma], X_j)$
where $X_j = X_t$, $X_R$, $X_{\gamma_0}$ or $X_\Gamma$.
\begin{prop}
  The quantity $\calF([\gamma], X_j)$ vanishes when $X = X_t$, $X_R$
  or $X_{\gamma_0}$, and is nonzero when $X = X_\Gamma$.
\end{prop}
\begin{proof}
  The vector field $X_t$ is odd with respect to the reflection ${\cal R}_t$;
  similarly, $X_R$ and $X_{\gamma_0}$ are odd with respect to one or
  more of the reflections ${\cal R}_o$, ${\cal R}_s$. Since the choice of generator
  $\gamma$ for $H_1$ is invariant under all three reflections, it is
  easy to see that $\calF([\gamma], X) = 0$ when $X$ is any one of
  these three vector fields. However, $X_\Gamma$ is a positive
  multiple of $\nu$ at every point of $\gamma$, so that
  $\calF([\gamma], X_\Gamma) > 0$, as claimed.

  We do not actually compute the value of this one nonvanishing flux.
\end{proof}

Unlike many other gluing constructions for minimal surfaces, these
fluxes turn out to play no interesting role in the analysis
below. This traces, ultimately, to the fact that we will be gluing
together copies of horizontal catenoids and these are already
`balanced'.  We explain this point further at the end of \S 5.

\medskip

\noindent{\bf Spectrum of the Jacobi operator:} 
We now study the $L^2$ spectrum of the Jacobi operator $L_\eta$.  By the general considerations described
above, 
\[
\mbox{spec}(-L_\eta) = \{\lambda_j(\eta)\}_{j=1}^N \cup [1,\infty).
\]
The ray $[1,\infty)$ consists of absolutely continuous spectrum (this is because $K_\eta$ is a decaying
perturbation of the union of two planes outside a compact set, so that the essential spectrum of
$-L_\eta$ coincides with that of $-L_P$), while the discrete spectrum lies entirely in $(-\infty, 1)$; note that,
even counted according to multiplicity, the number of eigenvalues may depend on $\eta$.

Our main result is the following: 
\begin{prop}
  For each $\eta \in (0,\eta_0)$, the only one of the eigenvalues of
  $-L_\eta$ which is negative is $\lambda_0(\eta)$, and only $
  \lambda_1(\eta) = 0$. All the remaining eigenvalues are strictly
  positive. The ground-state eigenfunction $\phi_0 = \phi_0(\eta)$ is
  even with respect to all three reflections, ${\cal R}_t$, ${\cal
    R}_s$ and ${\cal R}_o$; the eigenfunction $\phi_1$, which is the
  unique $L^2$ Jacobi field, is generated by vertical translations and
  is odd with respect to ${\cal R}_t$ but even with respect to ${\cal
    R}_s$ and ${\cal R}_o$.  In particular, if we restrict $-L_\eta$
  to functions which are even with respect to ${\cal R}_t$, then
  $L_\eta$ is nondegenerate.
\label{specnondeg}
\end{prop}
\begin{proof}
  We can decompose the spectrum of $-L_\eta$ into the parts which are
  either even or odd with respect to each of the isometric reflections
  ${\cal R}_t$, ${\cal R}_s$ and ${\cal R}_o$. Indeed, for each such
  reflection, there is an even/odd decomposition
\[
L^2(K_\eta)  =  L^2(K_\eta)_{j-\mathrm{ev}} \oplus L^2(K_\eta)_{j-\mathrm{odd}},\ j = t, s, o. 
\]
The reduction of $-L_\eta$ to the odd part of any one of these decompositions corresponds to 
this operator acting on functions on the appropriate half $K_\eta^{j,+}$ of $K_\eta$ with Dirichlet boundary
conditions. 

Our first claim is that the restriction of $-L_\eta$ to $L^2(K_\eta)_{j-\mathrm{odd}}$ with $j = s, o$ is
strictly positive, and is nonnegative if $j = t$, with one-dimensional nullspace spanned by the Jacobi field $\Phi_t$ 
generated by vertical translations.

To prove this, note first that since $\Phi_t \in L^2(K_\eta)_{t-\mathrm{odd}}$ and $\Phi_t$ is strictly positive
on $K_\eta^{t,+}$, it must be the ground state eigenfunction for this reduction and is thus necessarily
simple, with all the other eigenvalues strictly positive.

On the other hand, we have proved above that $\Phi_s$ and $\Phi_o$ are strictly positive solutions
of this operator on the appropriate halves of $K_\eta$, vanishing on the boundary, but of course 
do not lie in $L^2$.  We shall invoke the following Lemma. 
\begin{lem}
Consider the operator $-L=-\Delta + V$ on a Riemannian manifold $M$, where $V$ is smooth and bounded. 
Assume either that $M$ is complete, or else, if it has boundary, then we consider $-L$ with Dirichlet
boundary conditions at $\del M$. Suppose that there exists an $L^2$ solution $u_0$ of
$L u_0 = 0$ such that $u_0 > 0$, at least away from $\del M$. If $v$ is any other solution
of $L v = 0$ with $v > 0$ in $M$ and $v = 0$ on $\del M$, then $v = c u_0$ for 
some constant $c$. 
\label{gclem}
\end{lem}
\begin{rem}
We can certainly relax the hypotheses on $V$. The proof below is from the paper of Murata \cite{Mu}; the 
result appears in earlier work by Agmon, and is proved by different methods in \cite[Theorem 2.8]{Sullivan}
and \cite[Ch. 4, Theorem 3.4]{Pinsky}
\end{rem}
\begin{proof} 
It is technically simpler to work on a compact manifold with smooth boundary, so let $\Omega_j$ be a sequence of nested, 
compact smoothly bounded domains which exhaust $M$, and in the case where  $\del M \neq \emptyset$, assume that 
$\overline{\Omega_j} \cap \del M = \emptyset$ for all $j$.  The last condition is imposed since it is convenient
to have that $v$ is strictly positive on the closure of each $\Omega_j$.

It is well-known that the lowest eigenvalue $\lambda_0^j$ of $-L$ with Dirichlet boundary conditions on $\Omega_j$ converges
to the lowest Dirichlet eigenvalue $\lambda_0$ of $-L$ on all of $M$ (indeed, this follows from the Rayleigh quotient characterization
of the lowest eigenvalue).  We are assuming that $\lambda_0 = 0$, so by domain monotonicity, $\lambda_0^j \searrow 0$. 

Now choose a nonnegative (and not identically vanishing) function $\psi \in \calC^\infty_0(\Omega_0)$ and define 
$-L_k = -L - \frac{1}{k}\psi$ for any $k \in \RR^+$.  Denoting the lowest eigenvalue of this operator on $\Omega_j$
by $\lambda_0^{j,k}$, then by the same Rayleigh quotient characterization, we have that
$\lambda_0^{j,k} \leq \langle -L_k u, u \rangle$ for any fixed $u \in H^1_0(\Omega_j)$ with $||u||_{L^2} = 1$. In particular, inserting
the ground state eigenfunction $\hat{u}_0^j$ for $-L$ on $\Omega_j$, we obtain
\[
\lambda_0^{j,k} \leq \lambda_0^j - \frac{1}{k} \int_{\Omega_j} \psi |\hat{u}_0^j|^2\, dV_g.
\]
In particular, fixing $k > 0$, then since the first term on the right can be made arbitrarily close to $0$
by assumption, we can choose $j$ so that $\lambda_0^{j-1,k} > 0$ and $\lambda_0^{j,k} \leq 0$. This is because
the integral in the second term on the right is bounded away from zero, which holds because $\hat{u}_0^j \leq \hat{u}_0^{j+1}$
on the support of $\psi$ (this can be proved using the maximum principle for $-L - \lambda_0^{j+1}$ to 
compare $\hat{u}_0^j$ and $\hat{u}_0^{j+1}$ on the smaller domain $\Omega_j$).  If we recall also that 
the eigenvalue $\lambda_0^{j,k}$ depends continuously (in fact, analytically) on $k$, then we can adjust
the value of $k$ slightly to a nearby value $k_j$ so that $\lambda_0^{j,k_j} = 0$.  Clearly $k_j \to \infty$. 
We have thus obtained a solution $u_0^j > 0$ of $-L_{k_j} u_0^j = 0$ on $\Omega_j$ with $u_0^j = 0$ on $\del \Omega_j$. 

Since the solution $v$ is strictly positive, we have that $\Delta \log v = V - |\nabla \log v|^2$. Now, using
that $u_0^j$ vanishes on $\del \Omega_j$, we compute that 
\begin{multline*}
\int_{\Omega_j} \left|\nabla\left(u_0^j/v\right)\right|^2 v^2 \, dV_g = \int_{\Omega_j} |\nabla u_0^j|^2 - \nabla (u_0^j)^2 \cdot
\nabla \log v  + (u_0^j)^2 |\nabla \log v|^2 \, dV_g \\
= \int_{\Omega_j} |\nabla u_0^j|^2 + (u_0^j)^2( V - |\nabla \log v|^2) + (u_0^j)^2 |\nabla \log v|^2 \, dV_g  \\
= \int_{\Omega_j} u_0^j (-\Delta + V) u_0^j \, dV_g = \frac{1}{k_j}   \int_{\Omega_j} \psi (u_0^j)^2 \, dV_g.
\end{multline*}
Normalizing so that $||u_0^j||_{L^2} = 1$, then it is straightforward to show that $u_0^j \to u_0$ on
any compact subdomain of $M$. Since the right hand side of this equation tends to $0$, so does the left, hence
in particular the integral of $|\nabla( u_0/v)|^2$ over any fixed $\Omega_{j'}$ vanishes, i.e.\ $v = c u_0$ as claimed. 
\end{proof}

This Lemma implies that it is impossible for $-L_\eta$ to have lowest
eigenvalue equal to $0$ on either of the subspaces
$L^2(K_\eta)_{j-\mathrm{odd}}$, $j = s, o$, since if this were the
case, then we could use the corresponding eigenfunction as $u_0$ in
Lemma~\ref{gclem} and let $v = \Phi_j$ to get a contradiction since
$\Phi_j \notin L^2$.

We shall justify below that when $\eta$ is very close to its maximal
value $\eta_0$, the lowest eigenvalue of $-L_\eta$ on
$L^2(K_\eta)_{j-\mathrm{odd}}$ is strictly
positive. 
Using the continuity of the ground state eigenvalue as $\eta$
decreases combined with the argument above, we see that this lowest
eigenvalue can never be negative on any one of these odd subspaces,
and the only odd $L^2$ Jacobi field is $\Phi_t$.  This proves the
claim.

We have finally reduced to studying the spectrum of $-L_\eta$ on
$L^2(K_\eta)_{\mathrm{ev}}$, i.e.\ the subspace which is even with
respect to all three reflections (we call this ``totally even'').
Because of the existence of an $L^2$ solution of $L_\eta u = 0$ which
changes signs, namely $u = \Phi_t$, we know that the bottom of the
spectrum of $-L_\eta$ is strictly negative, and we have proved above
that the corresponding eigenfunction must live in the totally even
subspace. (This is also obvious because of the simplicity of this
eigenspace and the fact that the corresponding eigenfunction is
everywhere positive.)  Thus $\lambda_0(\eta) < 0$ as claimed.

Now suppose that the next eigenvalue $\lambda_1(\eta)$ lies in the
interval $(\lambda_0(\eta), 0]$, and if $\lambda_1(\eta) = 0$, assume
that there exists a corresponding eigenfunction which is totally even.
Since this is the {\it second} eigenvalue, we know that the
corresponding eigenfunction $\phi_1(\eta)$ has exactly two nodal
domains.  However, it is straightforward to see using the symmetries
of $K_\eta$ that if $\phi$ is any function on $K_\eta$ which is
totally even and changes sign, then it cannot have exactly two nodal
domains.  Indeed, if that were the case, then the nodal line $\{\phi =
0\}$ would have to either be a connected simple closed curve or else
two arcs, and these would then necessarily be the fixed point set of
one of the three reflections. This is clearly incompatible with $\phi$
being totally even.

We are almost finished. It remains finally to prove that the lowest
eigenvalue of $-L_\eta$ on any one of the odd subspaces is nonnegative
when $\eta$ is sufficiently large.

As a first step, we first prove that $\lambda_0(\eta) \nearrow 0$ as
$\eta \nearrow \eta_0$.  Recall that in this limit, $K_\eta$ converges
(once we translate vertically by an appropriate distance) to the
limiting Scherk surface $K_{\eta_0}$. Moreover, $K_{\eta_0}$ is
strictly stable because the Jacobi field $\Phi_t$ generated by
vertical translation is strictly positive on it. 

Now suppose that $\lambda_0(\eta) \leq -c < 0$. When $\eta$ is
sufficiently close to $\eta_0$, we can construct a cutoff
$\tilde{\phi}_0(\eta)$ of the corresponding eigenfunction
$\phi_0(\eta)$ which is supported in the region $t > 0$ (we are still
assuming that $K_\eta$ is centered around $t=0$); this function lies
in $L^2$ and regarding it as a function on $K_{\eta_0}$, it is
straightforward to show that
\[
\frac{\int_{K_{\eta_0}} (-L_{\eta_0} \tilde{\phi_0} ) \tilde{\phi_0}
}{ \int_{K_{\eta_0}} |\tilde{\phi_0}|^2} \leq -c/2 < 0.
\]
This contradicts the strict stability of $K_{\eta_0}$, and hence
proves that $\lambda_0(\eta) \nearrow 0$.

Now suppose that there is some sequence $\eta^\ell \nearrow \eta_0$
and a corresponding sequence of eigenvalues $\lambda^\ell \in
(\lambda_0(\eta^\ell), 0)$ and eigenfunctions $\phi^\ell\in
L^2(K_\eta)_{j-\mathrm{odd}}$, $j = s, o$.
We know that $\lambda^\ell \nearrow 0$.  Suppose that the maximum of
$|\phi^\ell|$ is attained at some point $p^\ell \in
K_{\eta^\ell}$. Normalize by setting $\hat{\phi}^\ell = \phi^\ell/\sup
|\phi^\ell|$ and take the limit as $\ell \to \infty$. Depending on the
limiting location of $p^\ell$, we obtain a bounded solution of the
limiting equation on the pointed Gromov-Hausdorff limit of the
sequence $(K_{\eta^\ell}, p^\ell)$.  There are, up to isometries, only
two possible such limits: either the limiting Scherk surface
$K_{\eta_0}$ or else a vertical plane $P = \gamma \times \RR$.
In the latter case, the limiting function $\hat{\phi}$ satisfies $L_P
\hat{\phi} = 0$. However, $L_P = \Delta_P - 1$ and it follows by
an easy argument using the Fourier transform on $P$ that there are no
bounded solutions of $L_P \hat{\phi} = 0$ on all of $P$, so this case cannot occur. 
Therefore, we have obtained a function $\hat{\phi}$ on $K_{\eta_0}$
which is a solution of the Jacobi equation there and which is
bounded. We now invoke Theorem 2.1 in \cite{MPR}, which is a result
very similar to Lemma~\ref{gclem}, but instead of assuming that $v$ is
positive, we assume instead that $v$ is bounded, and then conclude
that $v = c u_0$ where $u_0$ is the positive $L^2$ solution. The proof
proceeds by a somewhat more intricate cutoff argument than the one
above. In any case, this proves that $\hat{\phi}$ must equal the
unique {\it positive} $L^2$ Jacobi field on $K_{\eta_0}$, but this is
impossible because of the oddness of $\phi^\ell$ with respect to
either ${\cal R}_s$ or ${\cal R}_o$.

This completes the proof of the main Proposition.
\end{proof}

\section{Families of nearly minimal surfaces}
We now describe a collection of families of `nearly minimal' surfaces, exhibiting a wide variety of topological types.  
In the next section we prove that these can be deformed to actual minimal surfaces, at least when certain parameters 
in the family are sufficiently large. The geometry of each such configuration is encoded by a finite network of geodesic 
lines and arcs in $\HH^2$. Each complete geodesic $\gamma$ in this network corresponds to a vertical plane 
$P = P_\gamma = \gamma \times \RR$. The geodesic segments connecting these geodesic lines 
correspond to catenoidal necks connecting the associated vertical planes. The approximate minimal surfaces themselves
are constructed by gluing together horizontal catenoids. Thus we take advantage of the existence of these components, 
the existence of which already incorporates some of the nonlinearities of the problem; this is in lieu of working with the 
more primitive component set comprised of vertical planes and catenoidal necks. The parameter which measures
the `strength' of the interaction between these pieces is the distance between the finite geodesic segments. Once 
this distance is sufficiently large, we expect that the approximately minimal surface can be perturbed to be exactly minimal.  
We prove this here under one extra hypothesis, that the catenoidal necksizes remain bounded away from zero. The more
general case will be handled in a subsequent paper. The joint requirement that the distances between geodesic `connector' arcs 
be large and that the necksizes are bounded away from zero imposes restrictions which we describe below. 

We now describe all of this more carefully.  

\subsubsection*{Geodesic networks} An admissible geodesic network $\calF$ (see Fig. \ref{fig:two}) consists of a finite
set of (complete) geodesic lines $\Gamma = \{\gamma_\alpha\}_{\alpha \in A}$ and geodesic segments
$\calT = \{ \tau_{\alpha \beta}\}_{ (\alpha, \beta) \in A'}$ connecting various pairs of elements
in $\Gamma$. Here $A$ is some finite index set and $A'$ is a subset of $A \times A \setminus \mbox{diag}$
which indexes all `contiguous' geodesics, $\gamma_\alpha$ and $\gamma_\beta$
which are joined by some $\tau_{\alpha \beta}$. We now make various assumptions on these data
and set notation: 
\begin{itemize}
\item[i)] If $\alpha \neq \beta$, then $\mbox{dist}\,(\gamma_\alpha , \gamma_\beta) : = \eta_{\alpha \beta}
\in (0,\eta_0)$, where $\eta_0$ is the maximal separation between vertical planes which support a 
horizontal catenoid. 
\item[ii)] The segment $\tau_{\alpha \beta}$ realizes the distance $\eta_{\alpha \beta}$ between 
$\gamma_\alpha$ and $\gamma_\beta$, and hence is perpendicular to both these geodesic lines.
\item[iii)] Set $p_\alpha(\beta) = \tau_{\alpha \beta} \cap \gamma_\alpha$ and $p_\beta(\alpha) = 
\tau_{\alpha \beta} \cap \gamma_\beta$, and then define 
\[
D_\alpha = \min_{(\alpha \beta), (\alpha, \beta') \in A'} \{ \mbox{dist}(p_\alpha(\beta), p_{\alpha}(\beta'))\},
\ \ \mbox{and}\ \ D = \min_{\alpha} D_\alpha.
\]
This number $D$ is called the minimal neck separation of the configuration $\calF$. 
\item[iv)] We also write $\eta := \sup \eta_{\alpha \beta}$, and call it the maximal neck parameter.
\end{itemize}
\begin{figure}[htbp]
    \begin{center}
        \includegraphics[width=.58\textwidth]{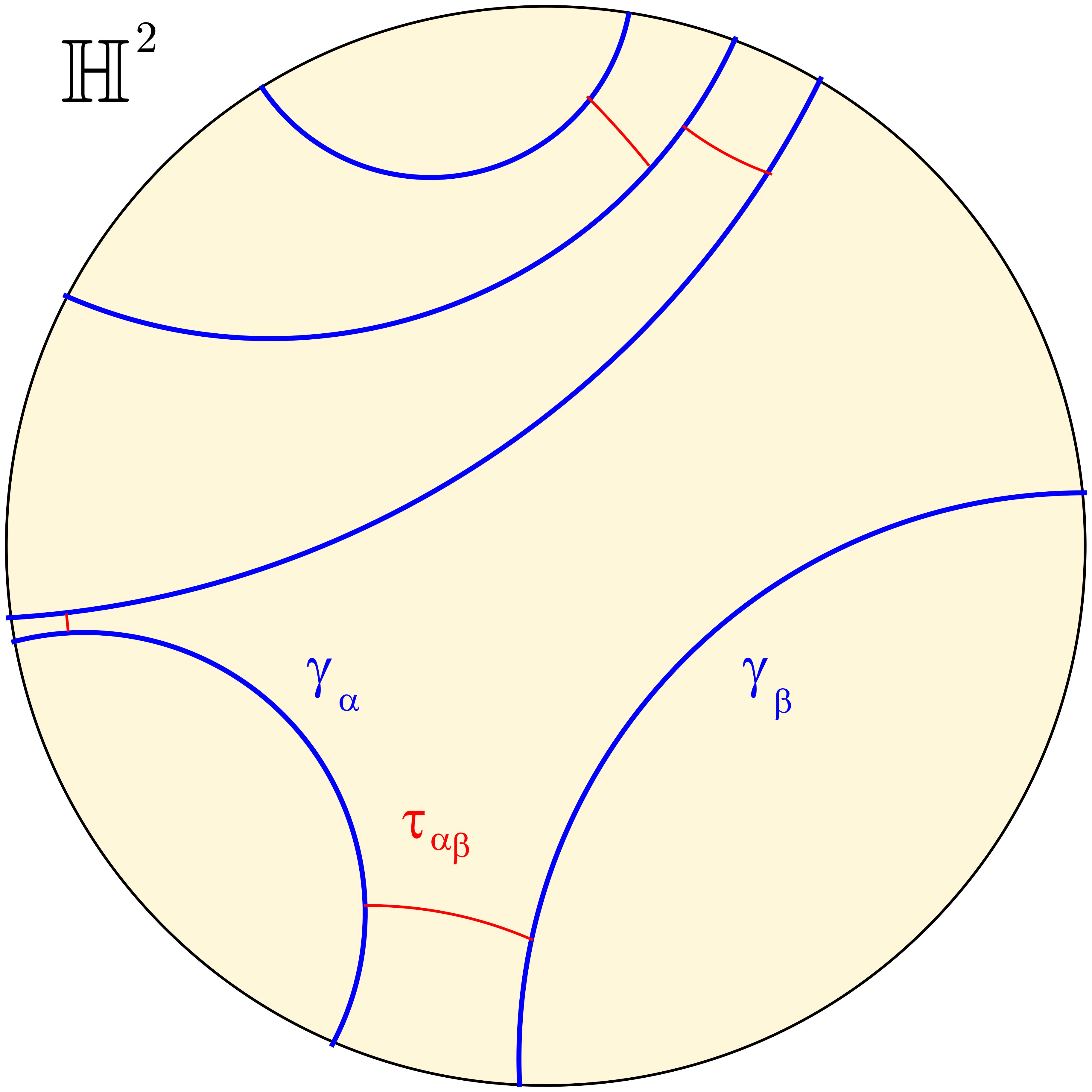}
     \end{center}
   \caption{The geodesic network $\cal F$.} \label{fig:two}
\end{figure}

We shall be considering sequences of geodesic networks $\calF_j$ for
which the minimal neck separation $D_j$ tends to infinity. Such
sequences have two distinct types of behaviour: either all of the
$(\eta_{\alpha \beta})_j \geq c > 0$, or else at least some of the
$(\eta_{\alpha \beta})_j \to 0$. The main analytic construction below
turns out to be fairly straightforward for the first type, but
unfortunately the simplest geometries (a relatively small number of
ends for a given genus) can only happen in the second setting.
\begin{prop}
  Let $\calF_j$ be a sequence of configurations with $D_j \to \infty$,
  and suppose that no $\calF_j$ is contractible.  If the cardinalities
  of the index sets $A(\calF_j)$ and $A'(\calF_j)$ (i.e.\ the number
  of geodesics and geodesic segments) remain bounded independently of
  $j$, then at least some of the necksizes $(\eta_{\alpha \beta})_j$
  must tend to $0$.
  \label{prop:necksizes}
\end{prop}
\begin{proof}
  By hypothesis, for each $j$ the configuration $\calF_j$ contains a
  cycle $c_j$. Referring to the geometry of each $\calF_j$, it is
  clear that each side of every $c_j$ is a geodesic segment, and
  moreover, each $c_j$ is a convex hyperbolic polygon whose sides meet
  at right angles.  By hypothesis then we have a sequence of such
  polygons where the number of sides remains bounded, so we may as
  well assume that each $c_j$ is a $k$-gon for some fixed $k$. Suppose
  that all $(\eta_{\alpha\beta})_j \geq c > 0$. Then by hypothesis,
  the successive adjacent sides of $c_j$ are geodesic segments of length at
  least $D_j$ and geodesic segments of length lying in the interval
  $[c, \eta_0]$. However, it is a standard fact in hyperbolic geometry
  that a geodesic polygon with every other side lying in such an
  interval must have all sidelengths uniformly controlled, which is a
  contradiction.
\end{proof}
In summary, for any sequence of configurations with fixed nontrivial
topology and a fixed number of geodesic lines , at least some of the
$\eta_{\alpha \beta}$ must converge to $0$.

\subsubsection*{From geodesic networks to nearly minimal surfaces}
To each geodesic network $\calF$ satisfying the properties above we
now associate an approximately minimal surface $\Sigma$.  The idea is
straightforward: each geodesic line $\gamma_\alpha$ is replaced by the
vertical plane $P_\alpha = \gamma_\alpha \times \RR$, and each
geodesic segment $\tau_{\alpha \beta}$ corresponds to a catenoidal
neck connecting $P_\alpha$ and $P_\beta$ at the points
$p_\alpha(\beta)$ and $p_\beta(\alpha)$.  The resulting surface is
denoted $\Sigma_{\calF}$.

The arguments used below to deform $\Sigma_{\calF}$ to an actual
minimal surface are perturbative, so we must construct sequences of
nearly minimal surfaces for which the error, which is a quantitative
measure of how far $\Sigma_{\calF}$ is from being minimal, tends to
zero.  To make the error term small, it is necessary to consider a
sequence of networks $\calF_j$ where the minimum neck separation $D_j$
tends to infinity. As proved above, if the necksizes stay bounded away
from zero, the number of component pieces must grow with $j$. Because
the proof is much cleaner in this case, we assume that $(\eta_{\alpha
  \beta})_j \geq c > 0$ in all the rest of this paper.  The more
general case can be treated using techniques similar to those in
\cite{MazPacPol}, but we shall address this in a separate paper.

The surface $\Sigma_{\calF}$ will be constructed by assigning to each
$\tau_{\alpha \beta}$ a vertical strip in the catenoid
$K_{\eta_{\alpha \beta}}$ (see below) which contains a very wide neighbourhood
around the neck region.  Using that none of the necksizes tend to
zero, we will prove that the Jacobi operator has a uniformly bounded
inverse, acting between certain weighted H\"older spaces.

\medskip For each $\calF$, we now show how to construct
$\Sigma_{\calF}$. Fix a line $\gamma_\alpha$ in $\calF$, and enumerate
the points $p_{\alpha}(\beta)$ along this line consecutively as
$p_{\alpha,1}, \ldots, p_{\alpha,N}$. (The number of such points, $N =
N_\alpha$, depends on $\alpha$, but for the sake of simplicity, the
notation does not record this.)  Let $q_{\alpha,j}$ be the midpoint of
the geodesic segment from $p_{\alpha,j}$ to $p_{\alpha,j+1}$, $j = 1,
\ldots, N-1$ and denote the length of such a segment by $d_{\alpha,
  j}$. Hence
\[
\mbox{dist}(p_{\alpha,j}, q_{\alpha,j}) = \frac12 d_{\alpha,j},\quad
\mbox{dist}(p_{\alpha,j}, q_{\alpha,j-1}) = \frac12 d_{\alpha,j-1}.
\]
Note that each $d_{\alpha, j} \geq D_\alpha > D$.  Finally, let
$S_{\alpha,j}$ denote the vertical strip in $P_\alpha$ bounded by the
two lines $q_{\alpha,j} \times \RR$ and $q_{\alpha,j+1} \times
\RR$. For the extreme values $j = 0$ and $N$, let $S_{\alpha,j}$ be
the half-plane in $P_\alpha$ bounded by $q_{\alpha,1}$ (on the right)
and $q_{\alpha,N-1}$ (on the left), respectively.

Now, consider a geodesic segment $\tau_{\alpha \beta} \in \calF$, and
write its two endpoints as $p_{\alpha,j}$ and $p_{\beta,k}$.  Let
$K_{\alpha \beta}$ be the horizontal catenoid with vertical ends
$P_\alpha \sqcup P_\beta$ and parameter $\eta_{\alpha \beta}$. Writing
this catenoid as a horizontal normal graph over the relevant portions
of the planes $P_\alpha$ and $P_\beta$ (i.e.\ away from the neck
regions), we let $K^c_{\alpha \beta}$ denote the portion of the
catenoid which includes the neck region and which lies over the strips
$S_{\alpha,j}$ and $S_{\beta,k}$ (this is possible when $D$ is large
enough).

This ensemble is not quite in final form since the edges of the
different truncated horizontal catenoids do not quite match up. Write
the corresponding portion of $K_{\alpha \beta}$ over the strip
$S_{\alpha, j}$ as a normal graph with graph function $v_{\alpha, j}$.
Choose a smooth cutoff function $\chi_{\alpha, j} \geq 0$ which equals
$1$ in the interior of $S_{\alpha,j}$ at all points which are a
distance at least $2$ from the boundaries, and which vanishes at all
points which are distance at most $1$ from these boundaries, and such
that $|\nabla \chi_{\alpha, j}| + |\nabla^2 \chi_{\alpha, j}| \leq 2$
(again this is possible for $D$ is large enough). We then let
$K_{\alpha \beta}^0$ be the slightly modified surface which agrees
with $K^c_{\alpha \beta}$ near the neck region and is the graph of
$\chi_{\alpha, j} v_{\alpha, j}$ over the rest of $S_{\alpha, j}$. Of
course, this is no longer quite minimal where the modifications have
been made.

\begin{figure}[htbp]
    \begin{center}
        \includegraphics[width=\textwidth]{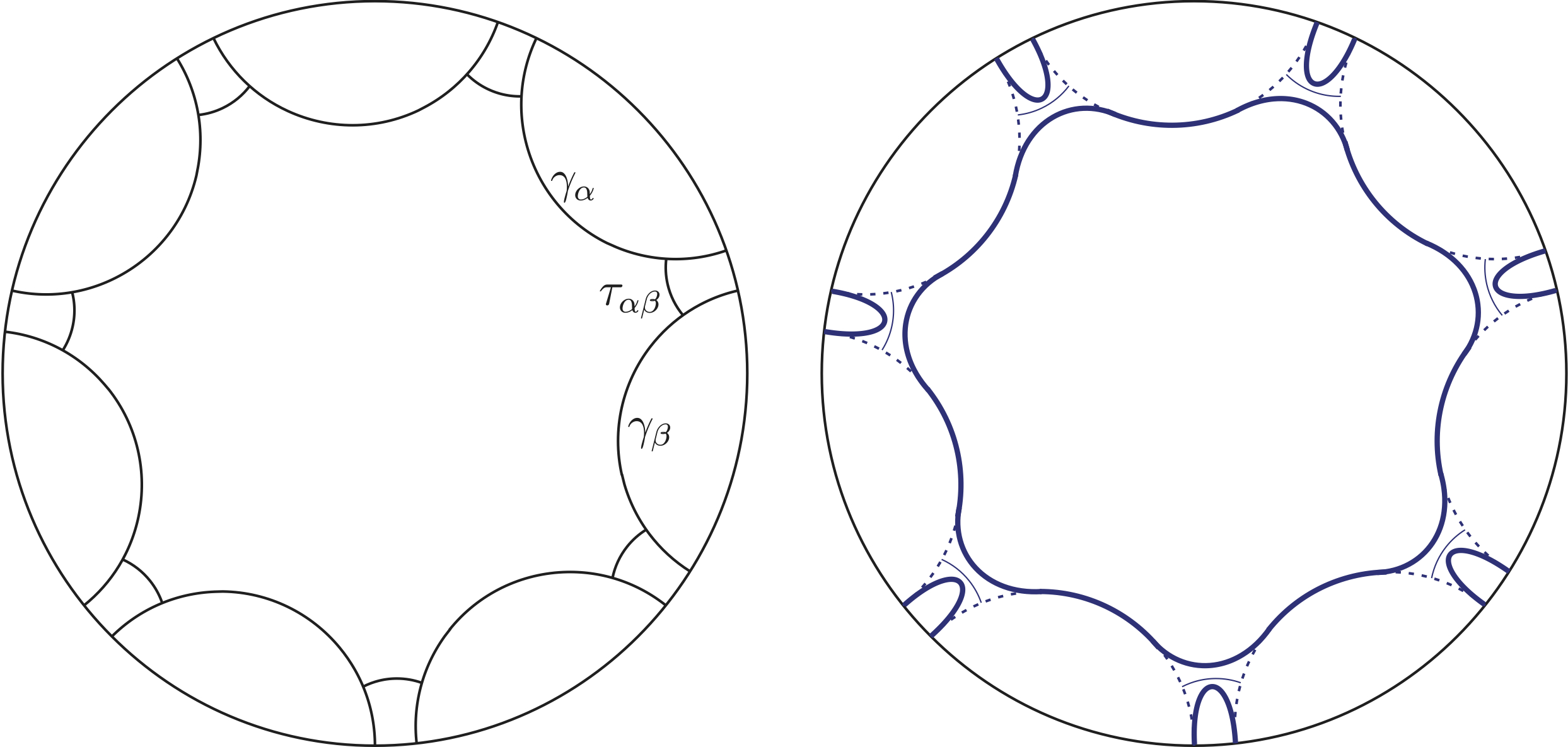}
     \end{center}
     \caption{An example of $\Sigma_{\cal F}$ and the corresponding (approximately) minimal surface}
     \label{fig:genus1}
\end{figure}

Our final definition of the approximately minimal surface
$\Sigma_{\calF}$ in this case, where all neck parameters are bounded
below by $c$, is
\begin{equation}
  \Sigma_\calF = \bigsqcup_{(\alpha \beta) \in A'} K_{\alpha \beta}^0.
  \label{defsigma}
\end{equation}

\begin{prop} \label{prop:H}
  Let $\calF$ be a geodesic network which satisfies the properties i), ii) and iii), and let $\Sigma = \Sigma_{\calF}$ 
be the associated   surface in $\HH^2 \times \RR$ just constructed.  Then $\Sigma$ is smooth and has $H \equiv 0$ 
except in the vertical strips $Q_{\alpha,j}$ of width $2$ around the lines $q_{\alpha,j} \times   \RR$. 
In the vertical strip $Q_{\alpha,j}$,
\[
\sup_{B_t} \|H\|_{0,\mu; B_t} \leq C r^{-\frac12} e^{-r};
\]
here $r = \min\{ \sqrt{ (d_{\alpha,j})^2 + t^2}, \sqrt{
  (d_{\alpha,j+1})^2 + t^2}\}$ and $B_t$ is the square of width $2$
and height $2$ centered at $(q_{\alpha,j},t) \in Q_{\alpha,j}$. The
constant $C$ is independent of all parameters in the construction
provided $D = \min D_\alpha$ is sufficiently large.
\end{prop}

The only point which needs to be checked is the decay of the local H\"older norm of the mean curvature.  
However, this follows directly from the corresponding estimate for the decay of the horizontal graph
functions $v_{\alpha,j}$, see~\eqref{genasym}.

\begin{prop}
\label{prop:index0} Let ${\cal F}$ be a geodesic network satisfying i), ii) and iii), and write $\Sigma=\Sigma_{\cal F}$. 
If $D$ is sufficiently large, then the Jacobi operator $L_\Sigma$ is non-degenerate when restricted 
to functions which are even respect to ${\cal R}_t.$
\end{prop}
\begin{proof}
We proceed by contradiction. Assume there exists a sequence of networks $\calF_j$ with $D_j=D(\calF_j) 
\nearrow +\infty$ such that each $\Sigma_j=\Sigma({\cal F}_j)$ admits a nonvanishing, even, $L^2$ Jacobi field $\phi_j$.
Let $p_j \in \Sigma_j$ be a point where $|\phi_j|$ attains its maximum, and set $a_j:= |\phi_j(p_j)|$. 
If $T_j$ is an isometry of $\HH^2 \times \R$ with $T_j(p_j) = (0,0)$, and $S_j=T_j(\Sigma_j)$, then 
$\psi_j=\left(\frac{1}{a_j} \phi_j \right) \circ T_j^{-1}$ is a Jacobi field on $S_j$ with $\sup|\psi_j| = 1$
attained at $(0,0)$. Passing to a subsequence, $S_j$ converges to a surface $S_\infty$, which is clearly
either a vertical plane or a horizontal catenoid $K_\eta$, for some $\eta \in (0,\eta_0)$, and $\psi_j$ converges 
to a nontrivial, bounded Jacobi field $\psi_\infty$ on $S_\infty$. However, it is clear that no such Jacobi
field exists on a vertical plane. Furthermore, the expansion \eqref{farfield} holds also on ends of $K_\eta$
and shows that a bounded Jacobi field must, in fact, lie in $L^2$ (this could also be proved more directly
using a variant of the proof of Proposition \ref{genasymprop}.  However, $\psi_\infty$ is the limit of functions 
invariant with respect to $\calR_t$, hence also has this property. This contradicts the vertical nondegeneracy 
of horizontal catenoids as proved in Proposition \ref{specnondeg}, and completes the proof. 
\end{proof}

\subsubsection*{Examples} 
It is possible to construct nearly minimal surfaces as sketched above,
assuming that all necksizes $\eta_{\alpha \beta}$ are bounded away
from $0$, with arbitrary genus, though possibly a large number of
ends. Since each plane $P_\alpha$ is diffeomorphic to a once-punctured
sphere, we see that $\Sigma_{\calF}$ is a connected sum of such
spheres, with one connection corresponding to each geodesic segment
$\tau_{\alpha\beta}$.

To carry out the perturbation analysis, we must consider networks $\calF$ with $D(\calF)$ sufficiently
large. As already explained, this imposes various restrictions.  For example, to find a
sequence of networks $\calF_j$ with precisely one loop, and with
$D(\calF_j) \to \infty$ and all $\eta_{\alpha \beta} \geq c > 0$,
standard formulas from hyperbolic trigonometry show that the number of
edges must grow with $j$.  One construction is to take a hyperideal
polygon in $\HH^2$ which is invariant with respect to rotation by
$2\pi/j$, by which we mean a collection of $j$ disjoint geodesics with
a cyclic ordering and such that the minimal distance between any pair
of adjacent geodesics is some fixed number $\eta$.  If $\eta$ does not
tend to zero, then the only way to have the minimal neck separation
tend to infinity is if $j \to \infty$ (see
Proposition~\ref{prop:necksizes}).  By contrast, we can find sequences
of such networks with $j=3$, for example, if we let $\eta \to 0$.

\section{Perturbation of $\Sigma_\calF$ to a minimal surface}
We now complete the perturbation analysis to show how to pass from the nearly minimal
surfaces $\Sigma_\calF$ to actual minimal surfaces in $\HH^2 \times \RR$ 
when $\calF$ is a geodesic network with minimal neck separation $D$ sufficiently
large, and with a uniform lower bound $\eta_{\alpha \beta} \geq c > 0$ on the neck parameters.

Fixing $\calF$, let $\Sigma = \Sigma_\calF$ and let $\nu$ be the unit normal on $\Sigma$ with
respect to a fixed orientation. For any $u \in \calC^{2,\mu}(\Sigma)$, consider the normal graph
over $\Sigma$ with graph function $u$, 
\[
\Sigma(u) = \{ \exp_{q} (u(q) \nu(q)),\ q \in \Sigma\}.
\]
Assuming that all $\eta_{\alpha \beta} \geq c > 0$, then there exists $C = C(c) > 0$ such that 
if $||u||_{2,\mu} < C$, then $\Sigma(u)$ is embedded. 

The surface $\Sigma(u)$ is minimal if and only if $u$ satisfies a
certain quasilinear elliptic partial differential equation, $\calN(u)
= 0$, which calculates the mean curvature of $\Sigma(u)$. (A similar
argument was considered in the proof of Proposition~\ref{genasymprop}
considering the vertical plane $P$ instead of $\Sigma$.) We do not
need to know much about $\calN$ except the following. If we write
$\calN(u) = \calN(0) + \left. D \calN\right|_0(u) + Q(u)$, then
\begin{itemize}
\item[i)] $\calN(0) = H_\Sigma$; 
\item[ii)] the linearization at $u=0$ is the Jacobi operator of $\Sigma$, 
\[
\left. D \calN\right|_0 = L_\Sigma = \Delta_\Sigma + |A_\Sigma|^2 + \Ric(\nu,\nu);
\]
\item[iii)] if $\epsilon$ is sufficiently small and $||u||_{2,\mu} < \epsilon$,  then
\[
||\calN(u)||_{0,\mu} \leq C \epsilon\ \ \mbox{and}\ \ 
||Q(u)||_{0,\mu} \leq C \epsilon^2.
\]
\end{itemize}

The equation $\calN(u) = 0$ is equivalent to
\begin{equation}
  L_\Sigma u = - H_\Sigma - Q(u).
\label{eq1}
\end{equation}
The strategy is now a standard one: we shall define certain weighted
H\"older spaces $X$ and $Y$, and first prove that $L_\Sigma: X \to Y$
is Fredholm.  A more careful analysis will show that, at least when
the minimal neck separation $D$ is sufficiently large, this map is
invertible and moreover its inverse $G_\Sigma: Y \to X$ has norm which
is uniformly bounded by a constant depending only on the lower bounds
$D$ for the minimal neck separation and $c$ for the maximal neck
parameter (see Proposition~\ref{invertible} and Corollary~\ref{cor:G}).
Given these facts, we then rewrite \eqref{eq1} as
\begin{equation}
u = - G_\Sigma ( H_\Sigma + Q(u)),
\label{eq2}
\end{equation}
and solve this equation by a standard contraction mapping argument. 

Somewhat remarkably, in this instance, this argument works almost
exactly as stated. The only subtlety is that we must restrict to
functions which are even with respect to the vertical reflection $t
\mapsto -t$, since this subspace avoids the exponentially decaying
element of the nullspace of the Jacobi operator on $\Sigma$.

The basic function spaces are standard H\"older spaces
$\calC^{k,\mu}(\Sigma)$ defined using the seminorm
\[
[ u ]_{0,\mu} = \sup_{z \neq z' \atop \mbox{dist}\,(z,z') \leq 1}
\frac{ |u(z) - u(z')|}{ \mbox{dist}\,(z,z')^\mu}.
\]
Although the result could be proved using these spaces alone, we can
obtain finer results by including a weight factor, which involves the
exponential of a piecewise radial function $R$. On each strip
$S_{\alpha,j}$, define a radial function $r_{\alpha,j} = \sqrt{s^2 +
  t^2}$, where $s$ is the arclength parameter along $\gamma_\alpha$
and $s=0$ corresponds to the point $p_{\alpha,j}$. The functions
$r_{\alpha,j}$ and $r_{\alpha,j+1}$ match up continuously at
$S_{\alpha,j} \cap S_{\alpha,j+1}$. Now define a function $R$ on
$\Sigma$ as follows: on each neck region of every horizontal catenoid
set $R \equiv 1$; on the portion of $\Sigma$ which is a graph over
$S_{\alpha,j} \setminus \calO_{\alpha,j}$ (where $\calO_{\alpha,j}$ is
some ball which is larger than the projection of the neck region), set
$R = r_{\alpha,j}$.  It is convenient to replace this function with a
slightly mollified version which is smooth everywhere, and which has
the property that $|\nabla R| + |\nabla^2 R| \leq 2$, so we assume
this is the case.  Finally, define
\[
e^{\kappa R} \calC^{k,\mu}(\Sigma) = \{ u = e^{\kappa R} v: v \in \calC^{k,\mu}(\Sigma) \}.
\]
Given $u \in e^{\kappa R} \calC^{k,\mu}(\Sigma),$ we consider the following norm:
\[ \|u\|_{k,\mu,\kappa}= \| e^{- \kappa R} \, u \|_{k,\mu}. \]
\begin{prop} 
Fix any $\kappa \in (-1,1)$ and $\mu \in (0,1)$. If $\Sigma$ is any nearly minimal surface, as 
constructed above, then 
\[
L_\Sigma: e^{\kappa R} \calC^{2,\mu}(\Sigma) \longrightarrow e^{\kappa R} \calC^{0,\mu}(\Sigma)
\]
is Fredholm.
\label{Fredholms}
\end{prop}
\begin{proof} If the elliptic operator $L_\Sigma$ has local
  parametrices with compact remainder on each end of $\Sigma$, then we
  can patch together these local parametrices to obtain a parametrix
  on all of $\Sigma$ with similarly good properties. Recall that a
  local parametrix on a bounded open set $\calU$ in $\Sigma$ is a
  continuous linear operator $\tilde{G}_\calU: \calE'(\calU) \to
  \calD'(\calU)$, between the spaces of compactly supported and all
  distributions on $\calU$, satisfying
\begin{equation}
  L \tilde{G}_\calU = \mbox{Id} -\mathfrak{R}, \quad \tilde{G}_\calU L
  = \mbox{Id} - \mathfrak{R}',
\label{defparam}
\end{equation}
where $\mathfrak{R}$ and $\mathfrak{R}'$ are smoothing of infinite
order, and such that
\[
\tilde{G}_\calU: \calC^{0,\mu} \cap \calE'(\calU) \longrightarrow \calC^{2,\mu}(\calU).
\]
Similarly, if $E$ is any infinite end of $\Sigma$, then a local parametrix on $E$ is a linear 
operator $\tilde{G}_E$ which is bounded as a map $e^{\kappa R} \calC^{0,\mu}(E) \to e^{\kappa R}
\calC^{2,\mu}(E)$, and satisfies the analogue of \eqref{defparam}, where $\mathfrak{R}$ and 
$\mathfrak{R}'$ are again infinite order smoothing operators which have range in a space 
of (smooth) functions which have a fixed rate of decay at infinity.  It follows directly from
the Arzela-Ascoli theorem and these mapping properties that $\mathfrak{R}$ and $\mathfrak{R}'$ 
are compact operators on these weighted H\"older spaces. Thus, once we produce this parametrix,
which is an inverse to $L$ modulo compact remainder terms, a standard argument from
functional analysis then shows that $L$ is Fredholm. 

Since $L_\Sigma$ is uniformly elliptic, the existence of local parametrices on bounded open sets is 
one of the basic theorems of microlocal analysis, see \cite{Shubin}. The construction of parametrices 
on the ends of $\Sigma$ uses more, namely that $L_\Sigma$ is `fully elliptic' near infinity, which 
means that it is strongly invertible there in a sense we make precise below. 

Each end of $\Sigma$ has the form $P \setminus \calO$ where $P$ is a vertical plane and $\calO$ is a 
large ball of finite radius.  The restriction of $L$ to each end is a decaying perturbation of the 
basic operator $\Delta_{\RR^2} - 1$. The restriction to the complement of $\calO$ of this operator 
has an inverse, the Schwartz kernel of which, also known as the Green function, is expressed in terms 
of the modified Bessel function
\[
G_{\RR^2}(z,z') = c K_0(|z-z'|) \sim c' |z-z'|^{-\frac12} e^{-|z-z'|}\
,\ \mbox{as}\ |z-z'| \to \infty.
\]
It is not hard to check (see \cite{MV}) that if $|\kappa| < 1$ and $r = |z|$, then
\[
G_{\RR^2}:  e^{\kappa r} \calC^{0,\mu}(\RR^2) \longrightarrow e^{\kappa r} \calC^{2,\mu}(\RR^2).
\]
(The fact that this operator increases regularity by $2$ orders is classical; the slightly more subtle point is
that it also preserves the growth or decay rate $e^{\kappa r}$ when $|\kappa| < 1$.) 
Now write $L_\Sigma = \Delta_{\RR^2} - 1 + F$ where $F$ is a second order operator with smooth
coefficients which decay like $e^{-r}$. From this we deduce that
\[
L_\Sigma G_{\RR^2} - \mbox{Id} = \mathfrak{R}: e^{\kappa r}
\calC^{0,\mu}(\RR^2 \setminus \calO) \longrightarrow e^{(\kappa - 1)
  r} \calC^{0,\mu}(\RR^2 \setminus \calO).
\]
This does not yet compactly include into $e^{\kappa r} \calC^{0,\mu}$ since there is no gain of regularity so
we cannot apply the Arzela-Ascoli Theorem. There are two effective ways to overcome this: first, restricting 
to the complement of an even larger ball, we can make the norm of this remainder term as small as desired,
hence $\mbox{Id} + \mathfrak{R}$ can be inverted using a Neumann
series. Equivalently, we can use a standard elliptic parametrix
construction to modify $G_{\RR^2}$ by an asymptotic series so that the
new modified parametrix satisfies $L G = \mbox{Id} - \mathfrak{R}$
where $\mathfrak{R}$ maps into $e^{(\kappa-1)r} \calC^\infty(\RR^2
\setminus \calO)$.  Either of these methods produces a global
parametrix for $L_\Sigma$ with compact remainder on each end of
$\Sigma$.

Now, cover $\Sigma$ by open sets of the form $P_\alpha \setminus
\calO_\alpha$ and one relatively compact open set $\calU$. Using the
standard elliptic parametrix construction on this bounded set and the
parametrices constructed above on each $P_\alpha$, we may form a
global parametrix as follows.  Choose a partition of unity for this
open cover, $\{\chi_0, \chi_\alpha\}_{\alpha \in A}$, and for each
open set here choose another smooth function $\tilde{\chi}_i$ with
support in $\calU$ for $i = 0$ and in $P_\alpha \setminus
\calO_\alpha$ for $i = \alpha$, such that $\tilde{\chi}_i = 1$ on the
support of $\chi_i$.  Now define
\[
\tilde{G}_\Sigma = \tilde{\chi}_0 G_0 \chi_0 + \sum_{\alpha \in A}
\tilde{\chi}_\alpha G_\alpha \chi_\alpha.
\]
We calculate that
\begin{multline*}
  L_\Sigma \tilde{G}_\Sigma = \tilde{\chi}_0 L_\Sigma G_0 \chi_0 +
  \sum_\alpha \tilde{\chi}_\alpha L_\Sigma G_\alpha \chi_\alpha +\\
  [L_\Sigma, \tilde{\chi}_0] G_0 \chi_0 + \sum_\alpha [L_\Sigma,
  \tilde{\chi}_\alpha] G_\alpha \chi_\alpha = \mbox{Id} +
  \mathfrak{R}_\Sigma.
\end{multline*}
We use here that $L_\Sigma G_i = \mbox{Id}$ on the support of $\chi_i$
so the first set of terms on the right is equal to $\sum
\tilde{\chi}_i \mbox{Id} \chi_i = \sum \chi_i \mbox{Id} =
\mbox{Id}$. The remainder $\mathfrak{R}_\Sigma$ is a pseudodifferential operator
of order $-1$ with image lying in the union of the supports of the
$\nabla \tilde{\chi}_i$, which is a compact set. Hence, using the well-known
mapping properties of such operators, $\mathfrak{R}_\Sigma: e^{\kappa R} \calC^{0,\mu} \to e^{\kappa R}
\calC^{0,\mu}$ is a compact operator.  A similar calculation shows
that $\tilde{G}_\Sigma L_\Sigma = \mbox{Id} + \mathfrak{R}_\Sigma''$
is also compact.

We have now produced an approximate inverse modulo compact remainders, and
as explained at the beginning of the proof, this suffices to prove that $L$ is Fredholm.
\end{proof}

The next step is to show that $L_\Sigma$ is invertible provided the
minimal neck separation $D$ is sufficiently large. This fails of
course if $L_\Sigma$ acts on the entire space $e^{\kappa
  R}\calC^{2,\mu}(\Sigma)$ because of the exponentially decaying
Jacobi field generated by vertical translations. To circumvent this
issue, we restrict $L_\Sigma$ to the subspace $e^{\kappa R}
\calC^{k,\mu}_{\even}(\Sigma)$ of even functions with respect to the
reflection $t \mapsto -t$. (Note that we can assume that the radial
function $R$ is even.)
\begin{prop}
  Let $\Sigma$ be a nearly minimal surface associated to the geodesic
  network $\calF$.  There exists a $D_0 > 0$ such that if the minimal
  neck separation $D$ is greater than $D_0$, 
  then
  \[
  L_\Sigma: e^{\kappa R}\calC^{2,\mu}_{\even}(\Sigma) \longrightarrow
  e^{\kappa R} \calC^{0,\mu}_{\even}(\Sigma)
  \]
  is invertible.
  \label{invertible}
\end{prop}
\begin{proof} We have already proved that the mapping $L_\Sigma$ is Fredholm on the entire 
weighted H\"older space, and it is clear that this remains true when restricting to the subspace 
of even functions. Let $G_\Sigma$ denote the generalized inverse of $L_\Sigma$. Recall that,
by definition, this means that $L_\Sigma G_\Sigma - \mbox{Id} = \mathfrak{R}_\Sigma$ is a projector
onto the complement of the range of $L_\Sigma$ and $G_\Sigma L_\Sigma - \mbox{Id} = \mathfrak{R}_\Sigma'$ 
is a projector onto the nullspace of $L_\Sigma$. In particular, these projectors both have 
finite rank. Since the index of $L_\Sigma$ vanishes, $\Tr \mathfrak{R}_\Sigma' - \Tr \mathfrak{R}_\Sigma =
\mbox{Ind}\, (L_\Sigma) = 0$.

To proceed, we sketch a slightly different version of the parametrix construction. Recall that $\Sigma$ 
is a union of truncated (and slightly perturbed) horizontal catenoids $K_{\alpha\beta}^0$. 
These catenoids are joined along the lines $\{q_{\alpha, j}\}\times \RR$, which are at distance at least 
$D/2$ away from each neck region. In fact, when $R >D/2$, Proposition \ref{prop:H} shows that 
\[
\sup e^{-\kappa R} |H_\Sigma| \leq C \sup e^{- (\kappa+1) R} R^{-1/2}
\leq C e^{-(\kappa+1)D/2} D^{-1/2}.
\]
On the other hand, this inequality trivially holds when $R \leq D/2.$

Now choose an open cover comprised by slightly larger truncations of
these catenoids, a partition of unity $\chi_{\alpha \beta}$ associated
to this open cover, and smooth cutoff functions $\tilde{\chi}_{\alpha
  \beta}$ which are supported in these same open sets and which equal
$1$ on the support of $\chi_{\alpha \beta}$.  Then set
\[
\hat{G}_\Sigma = \sum_{(\alpha \beta) \in A'}
\tilde{\chi}_{\alpha\beta} G_{\alpha \beta} \chi_{\alpha \beta}.
\]
Exactly the same computation as above shows that $L_\Sigma
\hat{G}_\Sigma - \mbox{Id} = -\hat{\mathfrak{R}}_\Sigma$ is compact on $e^{\kappa
  R} \calC^{0,\mu}_{\ev}(\Sigma)$, but furthermore has norm
$||\hat{\mathfrak{R}}_\Sigma|| \leq C e^{-(\kappa+1)D/2}$, where $C$
is independent of $D$.

Finally, choosing $D$ sufficiently large, then $\mbox{Id} -
\hat{\mathfrak{R}}_\Sigma$ is invertible on $e^{\kappa R}
\calC^{0,\mu}_{\ev}$, and hence $L_\Sigma G_\Sigma = \mbox{Id}$ where
\[
G_\Sigma = \hat{G}_\Sigma \circ (\mbox{Id} -
\hat{\mathfrak{R}}_\Sigma)^{-1}.
\]
This shows that $L_\Sigma$ is surjective.  The proof of Proposition \ref{prop:index0} 
applies equally well here and implies that $L_\Sigma$ is injective as well, provided $D$
is large enough. (Alternately, the index of $L_\Sigma$ is zero, hence injectivity
follows directly from surjectivity.) This concludes the proof.
\end{proof}

It is clear from the local nature of the H\"older norms and the definition of this
parametrix that the operator norm of $\hat{G}_\Sigma$ is uniformly bounded as $D \to \infty$.
Its modification by $(\mbox{Id} - \hat{\mathfrak{R}})^{-1}$ does not change this,
so we obtain the 
\begin{cor}
If $\Sigma$ satisfies all the assumptions of the previous proposition, then the norm of the 
inverse $G_\Sigma$ on $e^{\kappa R} \calC^{0,\mu}_{\ev}$ is uniformly bounded as $D \to \infty$.
\label{cor:G}
\end{cor}

The slightly surprising fact is that these estimates are independent of the topology or number of ends of $\calF$ 
and $\Sigma$, but this is due to the character of the function spaces being used. 

\begin{thm}  Let $\calF$ be a geodesic network in $\HH^2$ and $\Sigma_{\calF}$ the nearly minimal surface constructed 
from it. Also, fix $\kappa \in (-1,0)$. If the minimal neck separation $D$ is sufficiently large,
then there exists a function $u \in e^{\kappa R}\calC^{2,\mu}_{\ev}(\Sigma_\calF)$ with 
$||u||_{2,\mu,\kappa} \leq C e^{-(\kappa+1)D/2}D^{-1/2}$ such that $\Sigma_{\calF}(u)$ is
an embedded minimal surface which is a small normal graph over $\Sigma_{\calF}$.
\label{mgt}
\end{thm}
\begin{proof}
We solve $\calN(u) = 0$ in the function space $e^{\kappa R} \calC^{2,\mu}_{\ev}$ by rewriting
this equation as in \eqref{eq2}.  As $\kappa \in (-1,0)$, then 
\[
||Q(u)||_{0,\mu, \kappa} \leq C_1 ||u||_{2,\mu,\kappa}^2, 
\]
and hence if $||H_\Sigma||_{0,\mu,\kappa} \leq A$, then 
\[
|| G_\Sigma( H_\Sigma + Q(u))||_{2,\mu,\kappa} \leq C (A + C_1 ||u||^2_{2,\mu,\kappa}). 
\]
If $||u||_{2,\mu,\kappa} \leq \beta$, then the right hand side here is bounded by $C(A + C_1 \beta^2)$,
and $C(A + C_1 \beta^2) \leq \beta$ provided we choose $\beta = \lambda A$ for some large $\lambda$ 
and then let $A$ be very small.  With these choices, if we write \eqref{eq2} as $u = \calT(u)$, 
then $\calT$ maps the ball of radius $\beta$ in $e^{\kappa R}\calC^{2,\mu}_{\ev}$ to itself. A similar
analysis shows that $\calT$ is a contraction on this ball. 

This proves that there is a unique solution to $\calN(u) = 0$, and that $||u||_{2,\mu,\kappa} \leq \beta$.
Finally, since $\kappa < 0$, $|u| \leq \beta e^{\kappa R} \leq \beta$, and the derivatives of $u$
are similarly small, which implies that $\Sigma_{\calF}(u)$ is embedded.
\end{proof}

\begin{prop}
Let $\calF_j$ be a sequence of geodesic networks as in Theorem~\ref{mgt} (in particular the minimal neck separation $D(\calF_j) \to \infty$) and $\Sigma_j$ the corresponding
minimal surfaces. Suppose that the necksizes $(\eta_{\alpha \beta})_j$ in the constituent horizontal catenoids all
lie in a fixed interval $[c_1, c_2] \subset (0,\eta_0)$.  Then for $j$ sufficiently large, $\Sigma_j$ is horizontally nondegenerate. 
\end{prop}
\begin{proof}
  Suppose that this is not the case, so that there exists some
  subsequence $\Sigma_{j'}$, which we immediately relabel as
  $\Sigma_j$ and a function $\varphi_j \in L^2(\Sigma_j)$ which is
  even with respect to ${\cal R}_t$ and which lies in the nullspace of
  the Jacobi operator $L_j$ on $\Sigma_j$.  Renormalize $\varphi_j$ to
  have supremum equal to $1$, and suppose that this supremum is
  attained at a point $p_j \in \Sigma_j$.

At this point we can reason as in the proof of Proposition \ref{prop:index0}
to get a contradiction with the nondegeneracy of the vertical 
plane and the horizontal nondegeneracy of the horizontal catenoid. 
\end{proof}

\section{Gluing nondegenerate surfaces} 
A construction which is closely related to the one in the last section
is as follows.  Let $\Sigma_1$ and $\Sigma_2$ be two minimal surfaces
in $\HH^2 \times \RR$ with a finite number of vertical ends, each one
symmetric with respect to the reflection ${\cal R}_t$, and each one
horizontally nondegenerate.  Fix a vertical planar end $E_\ell \subset
\Sigma_\ell$, and choose a sequence of isometries $\phi_{\ell,j}$ (of
the form $\varphi_{\ell,j} \times \mbox{id}$ where each $\varphi_{\ell,j}$ is
an isometry of $\HH^2$) such that the surface $\Sigma_{\ell,j}:=\phi_{\ell,j} (\Sigma_\ell)$
converges to a fixed vertical plane $P = \gamma \times \RR$.
Parametrizing $\gamma$ as $\gamma(s)$, then we suppose that a
half-plane in the end $E_1$ in $\Sigma_{1,j}$ is a horizontal graph
over $(-B_{1,j},\infty) \times \RR$ with $B_{1,j} \to \infty$ and with
graph function $v_{1,j}$, while a half-plane $E_2$ in $\Sigma_{2,j}$
is a horizontal graph over $(-\infty, B_{2,j}) \times \RR$ with
$B_{2,j} \to \infty$ and with graph function $v_{2,j}$. We assume
finally that both $v_{\ell, j}$ converge to $0$ as $j \to \infty$.

Now let $\tilde{\Sigma}_{1, j}$ be the surface which agrees with $\Sigma_{\ell, j}$ away from the
half-plane $(-1,\infty) \times \RR$, and where the graph function is altered to $\chi_1(s) v_{1, j}$;
here $\chi_1(s)$ is a smooth monotone decreasing function which equals $1$ for $s \leq -1$ and 
vanishes for $s \geq 0$. We let $\tilde{\Sigma}_{2,j}$ be a similar alteration of $\Sigma_{2,j}$.
Finally, let 
\[
\Sigma(j) = \left(\tilde{\Sigma}_{1,j} \setminus((0,\infty) \times \RR)\right) \sqcup
\left(\tilde{\Sigma}_{2,j} \setminus((-\infty,0) \times \RR)\right).
\]
It is clear that $\Sigma(j)$ is exactly minimal outside of the vertical strip $(-1,1) \times \RR$.

Furthermore, it is clear that if $\Sigma_1$ and $\Sigma_2$ carry radial functions $R_1$ and $R_2$
as in the previous section, then we can form a radial function $R(j)$ on $\Sigma(j)$, and define
weighted H\"older spaces $e^{\kappa R(j)} \calC^{k,\mu}(\Sigma(j))$. In terms of these, the mean
curvature $H(j)$ of $\Sigma(j)$ tends to zero.

A straightforward modification of the arguments in the preceding section yield a proof of the
\begin{thm}
Let $\Sigma(j)$ be a sequence of nearly minimal surfaces, constructed as above. Assume (as stated
earlier) that both $\Sigma_1$ and $\Sigma_2$ are horizontally nondegenerate. Then for $j$
sufficiently large, there exists a function $u \in e^{\kappa R(j)} \calC^{2,\mu}(\Sigma(j))$ such
that the surface $\Sigma(j,u)$, which is the normal graph over $\Sigma(j)$ with graph function
$u$, is an embedded, horizontally nondegenerate minimal surface. 
\label{thm5.1}
\end{thm}

One must check first that $\Sigma(j)$ itself is nondegenerate for $j$ large, and then that the
norm of the inverse of its Jacobi operator on these weighted H\"older spaces remains uniformly
bounded as $j \to \infty$. These facts are both proved by contradiction, and the details of the 
proofs are very similar to what we have done above. The final step, using a contraction mapping
to produce the function $u$ whose graph is minimal, is again done as before.

Notice that if the genera of $\Sigma_1$ and $\Sigma_2$ are $g_1$ and
$g_2$, respectively, then $\Sigma(j)$ and hence the minimal surface
$\Sigma(j,u)$ has genus $g_1 + g_2$.

\medskip

\begin{cor}
  The construction in Theorem~\ref{thm5.1} can be continued
  indefinitely. In other words, let $\Sigma_\ell$ be an infinite
  sequence of minimal, horizontally nondegenerate surfaces, each with
  finite genus and finite number of planar ends, and let $P_j$ be one
  of the planar ends of $\Sigma_j$. Suppose that we have constructed a
  sequence of minimal, horizontally nondegenerate surfaces
  $\Sigma^{(N)}$ inductively by gluing $\Sigma_N$ to $\Sigma^{(N-1)}$
  with the end $P_N$ attached to the end corresponding to $P_{N-1}$ in
  $\Sigma^{(N-1)}$.  Then one can arrange the gluing parameters so
  that $\Sigma^{(N)}$ converges to a minimal surface with an infinite
  number of vertical planar ends.
\end{cor}
Indeed, each of the gluings here is given by Theorem~\ref{thm5.1}, so it remains only to show that one
can pass to the limit. 
For this, construct a sequence of properly embedded minimal surfaces $\{S_N\},$ and two sequences of 
positive real numbers $R_N \nearrow +\infty$ and $\varepsilon_N \searrow 0$ such that: 
\begin{enumerate}[(a)]
\item $S_N$ is obtained by gluing $\Sigma^{(N-1)}$ and $\Sigma^{(N)}$.
\item If $S_N$ is a normal graph of a function $u_N$, then $\|u_N
  \|_{2,\mu} \leq 2^{-N}$.
\item $S_N \setminus B(p_0,R_N)$ consists of (disjoint) neighborhoods
  of the ends of $S_N$, where $p_0$ is a fixed point in $\HH^2 \times
  \RR$.
\item For all $m \geq N$, we have that $S_m \cap B(p_0,R_N)$ lies on a
  $\varepsilon_N$-neighborhood of $S_N$ and can be written as a normal
  graph over $S_N$.
\end{enumerate}

The construction of such sequences is possible since we can choose the neck
separation parameter $D_N$ at the $N^{\mathrm{th}}$ stage sufficiently large. Thus 
it is clear (item (b)) that we can ensure that the sequence of normal graph functions $u_N$ 
converge locally uniformly in $\calC^\infty$ to a function which is uniformly small, so that
embeddedness is maintained (items (c) and (d)), and which decays exponentially along all ends.

This Corollary shows that there exist complete, properly embedded minimal surfaces in $\HH^2 \times \RR$ 
with vertical planar ends, with either finite or infinite genus and with an infinite number of ends.

We conclude this section with a brief remark concerning why the fluxes of horizontal 
catenoids, or of the more general constituent pieces considered in this section, play no role 
in this gluing construction. The reason is that we glue along vertical lines orthogonal to the axis
of the catenoid and positioned very far from it.  Although these lines are not closed, 
they are limits of a sequence of closed curves, namely rectangles lying over regions 
$\{S_1 \leq s \leq S_2; |t| \leq T\}$ where $S_2, T \nearrow \infty$. These rectangles 
are homologically trivial, so the flux over them vanishes, and hence the same is true over the
vertical lines. Because of this, there is no need to balance the fluxes of the summands in this construction against one another.

\section{Deformation theory}
We conclude this paper with a brief analysis of the moduli space of
even, properly embedded complete minimal surfaces with finite total
curvature in $\HH^2 \times \RR$.  Let $\calM_k$ denote the space of
all such surfaces with $k$ ends, each asymptotic to a vertical plane,
and which are symmetric with respect to the reflection ${\cal R}_t$.
\begin{thm}
The space $\calM_k$ is a real analytic set with formal dimension equal to $2k$. There is a stratum of $\calM_k$
consisting of horizontally nondegenerate elements which has dimension exactly equal to $2k$.
\end{thm}
\begin{remark}
This dimension count agrees with our construction: indeed, $2k$ is precisely the dimension of the space of admissible 
geodesic networks with $k$ geodesic lines, regardless of the number of `cross-piece' geodesic segments, since
in a given network $\calF$, each geodesic line $\gamma_\alpha$ has a two-dimensional deformation space, and 
any small perturbation of the geodesics uniquely determines the corresponding deformations of the geodesic 
segments $\tau_{\alpha \beta}$.   Note, however, that we are {\it not} demanding here that the minimal surfaces be 
ones that we have constructed. For example, it is conceivable that there exist surfaces whose necks are not 
centered on the plane of symmetry. This analysis of the deformation space is insensitive to this.

We do not factor out by the $3$ dimensional space of `horizontal' isometries of $\HH^2 \times
\RR$. But if we do this, then the dimension count $2k-3$ agrees with the dimension of the family 
of minimal surfaces in \cite{moro1}.
\end{remark}
\begin{proof}
The proof is very similar to the ones in \cite{MPU} and \cite{KMP} (and in several places since then), so we shall
be brief.  A different approach to the moduli space theory -- for minimal surfaces with finite
total curvature and parallel ends -- appears in \cite{PR}, but that relies on a Weierstrass representation 
which is not available here.

Fix $\Sigma \in \calM_k$ and enumerate its vertical planar ends as $\{P_\alpha\}_{\alpha \in A}$, so  
each $P_\alpha = \gamma_\alpha \times \RR$. For any sufficiently small $\e_{\alpha,j} \in \RR$, $j = 1,2$, 
we can deform $\gamma_\alpha$, and hence $P_\alpha$, by displacing the two endpoints of $\gamma_\alpha$ 
by these amounts, respectively (relative to a fixed metric on $\mathbb{S}^1$). Thus small deformations of the entire 
ensemble of vertical planes are in correspondence with $2k$-tuples $\e = (\e_{\alpha, 1}, \e_{\alpha, 2})_{\alpha \in A}$
with $|\e| \ll 1$. 

For each such $\e$, let $\Sigma(\e)$ denote a small deformation
$\Sigma(\e)$ of the surface $\Sigma = \Sigma(0)$, constructed as
follows. For each $\alpha$, write the end $E_\alpha$ of $\Sigma$ as a
normal graph over some exterior region $P_\alpha \setminus O_\alpha$
with graph function $v_\alpha$ defined in polar coordinates for $r
\geq R_0$.  Rotate $P_\alpha$ by the parameters $\e_\alpha$ to obtain
a new vertical plane $P_\alpha(\e)$.  Using the same graph function
$v_\alpha$, now defined on an exterior region in $P_\alpha(\e)$, we
obtain the deformed end $E_\alpha(\e)$; this is quite close to the
original end $E_\alpha$ over the annulus $\{R_0+1 \leq r \leq
R_0+2\}$, so we can write $E_\alpha(\e)$ as the graph of a function
$v_{\alpha,\e}$ defined on this annulus in the {\it original} plane
$P_\alpha$.  Finally, use a fixed cutoff function $\chi_\alpha$ to
define $\tilde{v}_{\alpha,\e} = \chi_\alpha v_\alpha + (1-\chi_\alpha)
v_{\alpha,\e}$ so that the graph of this new function agrees with the
original surface $\Sigma$ for $r \leq R_0 + 1$ and matches up smoothly
with $E_\alpha(\e)$ outside this annulus. This defines
$\Sigma(\e)$. Denoting its mean curvature function by $H(\e)$, then
clearly $H(\e)$ vanishes outside the union of these annuli, hence
$H(\e) \to 0$ in $e^{\kappa R} \calC^{2,\mu}(\Sigma(\e))$ as $|\e| \to
0$.

The remainder of the proof follows the corresponding arguments in \cite{MPU} and \cite{KMP} essentially verbatim. 
When $\Sigma$ is horizontally nondegenerate, the implicit function theorem produces an analytic function 
$\e \mapsto u_\e$ such that the normal graph of $u_\e$ over $\Sigma(\e)$ is minimal. This is a real analytic
coordinate chart in $\calM_k$ around $\Sigma$.  If $\Sigma$ is horizontally degenerate, then we can apply
a Lyapunov-Schmidt reduction argument to show that there exists a neighbourhood $\calU$ of $\Sigma$ in some 
fixed finite dimensional real analytic submanifold $Y$ in the space of all surfaces (with a fixed weighted H\"older regularity) 
and a real analytic function $F: \calU \to \RR$ such that $\calM_k \cap \calU = F^{-1}(0) \cap \calU$. 
\end{proof}

\bibliographystyle{plain}

\begin{thebibliography}{10}

\bibitem{And} {\em M. T.  Anderson},
  Complete minimal varieties in hyperbolic space.
  {\em Invent. Math.} {\bf 69} No. 3 (1982) 477--494. 

\bibitem{cor2}
{\em P.~Collin and H.~Rosenberg},
\newblock Construction of harmonic diffeomorphisms and minimal graphs.
\newblock {Annals of Math.}, {\bf 172} (2010), 1879--1906.

\bibitem{da2}
{\em B.~Daniel},
\newblock Isometric immersions into {$\mathbb{S}^n\times\mathbb{R}$} and
{$\mathbb{H}^n\times\mathbb{R}$} and applications to minimal  surfaces,
\newblock {Trans. Amer. Math. Soc.}, {\bf 361} (2009), 6255--6282.
\newblock MR2538594, Zbl pre05638191.


\bibitem{gt}
{\em D. Gilbarg and N. S. Trudinger},
\newblock Elliptic partial differential
equations of second order,
\newblock {Reprint of the 1998 edition. Classics in
Mathematics. Springer-Verlag, Berlin, 2001.}
\newblock MR1814364.


\bibitem{hst1}
{\em L.~Hauswirth, R.~Sa Earp, and E.~Toubiana},
\newblock Associate and conjugate minimal immersions $\HH^2\times \RR$, 
\newblock {Tohoku Math. J.}, {\bf 60} (2008), 267--286.
\newblock MR2428864, Zbl 1153.53041.

\bibitem{hnst}
  {\em L.~Hauswirth, B.~Nelli, R.~Sa Earp and E.~Toubiana},
  \newblock Minimal ends in $\HH^2 \times \RR$ with finite total curvature and a Schoen type theorem.
  \newblock Preprint  arXiv:1111.0851.

\bibitem{JS}
{\em H. Jenkins and J. Serrin}, 
The Dirichlet problem for the minimal surface equation, with infinite data,
 Bull. Amer. Math. Soc. {\bf 72} (1966) 102-106.
 
 \bibitem{KMP} {\em  R. Kusner, R. Mazzeo and D. Pollack}
 \newblock The moduli space of complete embedded constant mean curvature surfaces, 
 \newblock Geometric and Functional Analysis, {\bf 6} No. 1, (1996), 120-137.
 
 \bibitem{Lebedev} {\em N.N. Lebedev}
 \newblock Special functions and their applications.
 \newblock Dover Publications Inc., N.Y. 1972.

\bibitem{MPR}
{\em J.M.~Manzano, J.~P\'erez and M.M.~Rodr\'\i guez},
\newblock Parabolic stable surfaces with constant mean curvature,
\newblock {Calc. Var. PDEs}, {\bf 42} (2011), 137--152.
\newblock DOI: 10.1007/s00526-010-0383-6.

\bibitem{marr1}
{\em L.~Mazet, M.M.~Rodr\'\i guez, and H.~Rosenberg},
\newblock The {D}irichlet problem for the minimal surface equation with
  possible infinite boundary data over domains in a Riemannian   surface,
\newblock Proc. London Math. Soc., {\bf 102} (2011), 985-1023.
\newblock DOI: 10.1112 /plms/ pdq032.

\bibitem{MazPacPol}
{\em R.~Mazzeo, F. Pacard and D. Pollack}, 
\newblock Connected sums of constant mean curvature surfaces in Euclidean $3$-space.
\newblock Journal f\"ur die Reine und Angewandte Mathematik, {\bf 536} (2001), 115-165.

\bibitem{MPU} {\em R. Mazzeo, D. Pollack and K. Uhlenbeck},
\newblock Moduli spaces of singular Yamabe metrics, 
\newblock Journal of the American Mathematical Society, {\bf 9} No. 2, (1996), 303-344.

\bibitem{MS} {\em R. Mazzeo, M. S\'aez},
\newblock  Multiple-layer solutions to the Allen-Cahn equation on hyperbolic space. 
\newblock To appear, Proc. Amer. Math. Soc.

\bibitem{MV} {\em R. Mazzeo, A. Vasy},
\newblock  Resolvents and Martin boundaries of product spaces.
\newblock Geometric and Functional Analysis, {\bf 12} No. 5 (2002), 1018-1079.

\bibitem{m-p-r} {\em W. H. Meeks III, J. Perez, and A. Ros},
  Bounds on the topology and index of classical minimal
surfaces, Preprint available at \newline 
http://www.ugr.es/local/jperez/papers/papers.htm.

\bibitem{Melrose} {\em R. B. Melrose}, 
\newblock Geometric Scattering Theory. 
\newblock Cambridge University Press, Cambridge (1995). 

\bibitem{mo}
{\em F.~Morabito},
\newblock A Costa-Hoffman-Meeks type surface in $\HH^2\times\RR$,
\newblock {Trans. Amer. Math. Soc.}, {\bf 363} (2011), 1--36.

\bibitem{moro1}
{\em F.~Morabito and M.M.~Rodr\'\i guez},
\newblock Saddle towers and minimal $k$-noids in $\HH^2\times\RR$,
\newblock {J. Inst. Math. Jussieu}, {\bf 11} (2012), 333--349.
\newblock DOI: 10.1017/S1474748011000107.

\bibitem{Mu}  M.~Murata,
\newblock Structure of positive solutions to $(-\Delta +V)u = 0$ in $\RR^n$,
\newblock Duke Math. J.
\newblock {\bf 53} (1986) No. 4, 869--943.

\bibitem{ner2}
{\em B.~Nelli and H.~Rosenberg},
\newblock Minimal surfaces in $\HH^2 \times \RR$,
\newblock {Bull. Braz. Math. Soc.}, {\bf 33} (2002), 263--292.
\newblock MR1940353, Zbl 1038.53011.

\bibitem{PR} {\em J. P\' erez and A. Ros,}
\newblock The space of complete minimal surfaces with finite total curvature as a Lagrangian submanifold,
\newblock Trans. A. M. S. {\bf 351} No.10 (1999), 3935-3952

\bibitem{Pinsky} {\em M. Pinsky,}
\newblock {\em Large deviations for diffusion processes}, Stochastic Analysis, Academic Press, New York, 1978, 271-283.

\bibitem{pyo1}
{\em J.~Pyo},
\newblock New complete embedded minimal surfaces in $\HH^2 \times \RR$, 
\newblock {Ann. Glob. Anal. Geom.}, {\bf 40} (2011), 167--176.

\bibitem{Shubin}
{\em M.A. Shubin},
\newblock Pseudodifferential operators and spectral theory, Second Ed.
\newblock Springer--Verlag, Berlin, 2001. 

\bibitem{Sullivan}
{\em D. Sullivan}
\newblock Related aspects of positivity in Riemannian Geometry
\newblock J. Differential Geometry, {\bf 25} (1987), 326-351.
\end{thebibliography}





\end{document}